\documentclass[11pt]{article}

\usepackage{amsmath,amssymb,amsthm} 
\usepackage{mathrsfs,dsfont}
\usepackage{bbm}
\usepackage{paralist}
\usepackage{hyperref}
\usepackage{graphicx}
\usepackage{subcaption}
\graphicspath{{figures/}}
\usepackage{color}
\usepackage{todonotes}
\usepackage{bm}
\usepackage{float}
\usepackage{natbib}
\usepackage{algorithmic}
\usepackage[ruled,vlined]{algorithm2e}
\usepackage{booktabs}
\setdefaultenum{(i)}{}{}{}
\usepackage[margin=1in]{geometry}

\numberwithin{equation}{section}

\renewenvironment{thebibliography}[1]{%
\begin{oldthebibliography}{#1}%
\setlength{\baselineskip}{.8em}
\linespread{.8}
\small
\setlength{\parskip}{0ex}%
\setlength{\itemsep}{.1em}%
}%
{%
\end{oldthebibliography}%
}

\theoremstyle{plain}
\newtheorem{theorem}{Theorem}[section]
\newtheorem{proposition}[theorem]{Proposition}

\newtheorem{corollary}[theorem]{Corollary}

\DeclareMathAlphabet\scr{U}{scr}{m}{n}
\SetMathAlphabet\scr{bold}{U}{scr}{b}{n}
  \DeclareFontFamily{U}{scr}{\skewchar\font'177}%
  \DeclareFontShape{U}{scr}{m}{n}{<-6>rsfs5<6-8>rsfs7<8->rsfs10}{}%
  \DeclareFontShape{U}{scr}{b}{n}{<-6>rsfs5<6-8>rsfs7<8->rsfs10}{}%

\theoremstyle{definition}
\newtheorem{remark}[theorem]{Remark}
\newtheorem{assumption}[theorem]{Assumption}

\def\A{\mathcal{A}}

\def\F{\mathcal{F}}

\def\P{\mathbb{P}}
\def\E{\mathbb{E}}
\def\R{\mathbb{R}}

\def\diag{\text{diag}}
\def\argmin{\text{argmin}}

\newcommand{\XS}[1]{{\color{violet} #1}}

\title{Dynamic Portfolio Optimization under CVaR Constraints}
\author{
Anran Hu\footnote{Columbia University, Department of Industrial Engineering and Operations Research; \texttt{ah4277@columbia.edu}. }\qquad
Silvana M. Pesenti\footnote{University of Toronto, Department of Statistical Sciences; \texttt{silvana.pesenti@utoronto.ca}. } \qquad
Xiaofei Shi \footnote{University of Toronto, Department of Statistical Sciences; \texttt{xf.shi@utoronto.ca}.}
}

\date{}
\begin{document}

\maketitle
\begin{abstract}
We study continuous-time dynamic portfolio optimization under a Conditional Value-at-Risk (CVaR) constraint on the investor’s terminal loss. 
For a general class of convex trading objectives, we exploit the auxiliary-threshold representation of CVaR to establish the existence of an optimal strategy and strong duality without requiring market completeness. 
These results motivate a dual-based nested bisection--golden-search algorithm over the threshold and Lagrangian multiplier, where the inner iterations reduce to standard unconstrained stochastic control problems. 
We prove that the resulting strategies converge to the optimal control as the number of iterations tends to infinity. 
Numerical experiments recover the Merton policy when the risk constraint is nonbinding.
When the constraint is binding, the optimal strategy becomes state dependent: the investor reduces risky exposure following adverse outcomes but preserves, and near maturity may increase, exposure following favorable outcomes.
Thus, a terminal CVaR constraint produces an asymmetric reallocation across
states rather than uniform de-risking. Nontraded endowment risk amplifies the
conservative adjustment, whereas price impact lowers desired positions
and adjustment speeds.
\end{abstract}

\vspace{0.5em}
\noindent\textbf{Keywords:} Dynamic portfolio optimization; Conditional
Value-at-Risk; Constrained stochastic control; Lagrangian duality; Incomplete markets 

\vspace{0.5em}
\noindent\textbf{JEL Classification:}  G11; C61; C63
\section{Introduction}\label{sec:introduction}

Financial investors seek to generate returns while managing the risks associated with their investment strategies. Merton's foundational analyses of continuous-time portfolio selection and the efficient portfolio frontier established canonical benchmarks for this trade-off \cite{merton1969lifetime,merton1972analytic}. Subsequent works have extended portfolio choice in several directions, including market models with price impact  (e.g.~\cite{tsoukalas2019execution,edirisinghe2023leveraged}), trading costs (e.g.~\cite{dai2024dynamic,muhlekarbe2025dynamic}), robust and risk-aware decision criteria (e.g.~\cite{jaimungal2022SIAM}), and benchmark-relative portfolio optimization (e.g.~\cite{pesenti2023SIAM}).

Downside-risk bounds are particularly relevant in portfolio selection when regulatory mandates or trading-desk requirements impose direct limits on losses in adverse states. Unlike conventional risk aversion, such bounds target the downside tail directly. 
Conditional Value-at-Risk (CVaR), also known as Expected Shortfall, is a downside-tail risk measure that averages losses in a prescribed tail of the loss distribution \cite{acerbi2002coherence,rockafellar2002conditional}. In static portfolio optimization, CVaR may either be minimized as the
risk objective \cite{rockafellar2000optimization} or imposed as a constraint on an otherwise
return-oriented objective \cite{krokhmal2002portfolio,alexander2004comparison}. 
Related discrete-time multiperiod formulations include
CVaR-based risk control in \cite{chen2005multiperiod} and dynamic
mean--CVaR portfolio selection with a focus on time consistency in
\cite{strub2019discrete}.
In the continuous-time setting, one strand of the literature studies the cases
where the problem can be recast as the choice of a replicable terminal payoff, due to additional assumptions such as market completeness \cite{he2015dynamic,gao2017dynamic}.
Another strand considers self-financing portfolio choice under
dynamically re-evaluated risk constraints and exploits specialized
growth-optimal or CRRA structures
\cite{pirvu2009maximizing,morenobromberg2013crra}.

A common way is to incorporate downside risk through a soft
penalty in the investor's objective. Although
convenient, such a formulation requires the investor to specify an
exogenous penalty coefficient, and an arbitrary choice of this
coefficient does not ensure compliance with a prescribed risk budget.
In many institutional, regulatory, and wealth-management applications,
however, the relevant mandate is an explicit upper bound on downside
risk. Motivated by this consideration, we treat the terminal CVaR bound
as an exogenous risk-management requirement rather than a preference
penalty incorporated into the performance objective. Hard-constraint
formulations also arise more broadly in stochastic control, including
problems with expectation and terminal-law constraints
(e.g.~\cite{pfeiffer2021duality,daudin2022optimal}), and in mean-field games
with state or other feasibility constraints
(e.g.~\cite{cannarsa2018existence,hu2025mean}).

In this paper, we study a general continuous-time dynamic portfolio
optimization problem subject to a hard CVaR constraint on terminal
loss. Our framework does not require self-financing wealth dynamics or
market completeness and accommodates cumulative external cash flows
and nontraded endowment risk. In this general setting, the
problem cannot be reduced to a static choice of terminal payoff:
nontraded risks need not be replicable, while general running costs
make the entire state--control path relevant. In the frictional
extensions, the current position becomes an additional state variable
and trading speed becomes the control, so the timing and speed of
portfolio adjustment also affect optimality. The optimal strategy must
therefore be characterized dynamically rather than recovered solely
from an optimal terminal payoff.

To analyze this general portfolio optimization problem under hard CVaR constraints, we use the
Rockafellar--Uryasev representation of CVaR to obtain a convex
primal--dual formulation over the adapted trading strategy and an
auxiliary scalar threshold. This formulation preserves the dynamic
nature of the portfolio problem while yielding a modular solution
approach based on standard unconstrained stochastic control problems.
It provides the basis for both our theoretical analysis and the
computational method developed below.

\paragraph{Contributions.}
Our contribution to this line of research is threefold.

First, we provide a convex-analytic treatment of continuous-time
portfolio optimization under a hard terminal CVaR constraint in a
market that may be incomplete. The investor minimizes a general convex
expected running and terminal cost over admissible adapted trading
strategies. The framework covers, among other examples,
linear--quadratic and CARA specifications, as well as extended-valued
CRRA and logarithmic criteria whenever the corresponding
finite-feasibility condition is satisfied. We show that the CVaR
threshold can be restricted to a common compact interval and establish
existence of a primal optimizer. Under strict convexity, the optimal
trading strategy is unique. We further establish existence of
Lagrangian minimizers for every multiplier. Under a Slater condition,
we prove strong duality, existence and an explicit bound for an
optimal multiplier, and recovery of a primal optimizer from the dual
problem. These results do not require market completeness.

Second, we exploit the resulting max--min dual structure to develop a
modular numerical method. For each fixed Lagrange multiplier and CVaR
threshold, the innermost problem is a standard unconstrained stochastic
control problem, allowing existing dynamic-programming, PDE, or
stochastic-control solvers to be used as an oracle. An inner
golden-section search selects the threshold, while an outer bisection
adjusts the multiplier using the CVaR constraint residual. We quantify
the error of the inner search and prove convergence of the combined
procedure when the inner accuracy is increased appropriately relative
to the outer bisection. In particular, the constraint residual
converges to zero and the objective value converges to the primal
optimum. Under strict convexity, the computed controls converge weakly
to the unique optimal strategy; strong convexity strengthens this
conclusion to strong convergence in $L^2_{\mathbb F}$.

Third, we use numerical experiments to examine the effect of a binding terminal CVaR constraint in four environments: a frictionless complete-market benchmark, an incomplete market with nontraded endowment risk, a model with quadratic trading-rate regularization, and a model with square-root price impact.
The first three settings remain within the convex structure motivating our analysis, while the square-root specification provides a robustness experiment beyond the setting covered by our convergence theory. 
Three empirical patterns emerge across the reported calibrations.
First, the CVaR constraint reshapes the terminal-wealth distribution asymmetrically, with the adjustment concentrated in the downside tail rather than taking the form of a uniform contraction. 
This illustrates the distinction from conventional risk aversion: CVaR targets adverse tail outcomes directly rather than penalizing risk throughout the distribution.
Second, the response combines an initially conservative adjustment with state-dependent feedback.
Although the constraint is imposed only at maturity, when it binds, current wealth becomes a relevant state variable for the exposure policy: the computed strategies reduce risky exposure following adverse outcomes while maintaining greater participation following favorable outcomes.
Third, in the reported zero-interest, no-positive-inflow experiments, the unconditional loss-CVaR diagnostic remains below the terminal risk limit at every reported intermediate date across all four environments. 
Although the constraint is imposed only at maturity, the behavior of the unconditional interim CVaR diagnostic suggests that, in these experiments, the terminal constraint may also discipline risk taking earlier in the investment horizon.

More broadly, our work is connected to the literature on portfolio optimization under alternative downside-risk and distributional constraints.
Related formulations impose VaR constraints
\cite{basak2001var,pirvu2007portfolio,cuoco2008risklimits}
or the closely related Capital-at-Risk constraint
\cite{emmeretal2001bounded}, formulate portfolio choice in terms of wealth quantiles
\cite{he2011portfolio}, employ dynamic multivariate risk measures
\cite{feinstein2017recursive}, or impose utility-based shortfall-risk constraints
\cite{gundel2008shortfall}.
Our analysis also relates to benchmark-relative portfolio formulations that restrict terminal payoffs through divergence constraints, including Bregman--Wasserstein constraints
\cite{pesenti2024BW}
and their asymmetric $\alpha$-Bregman--Wasserstein extension
\cite{pesenti2026WP-alphaBW}.

The remainder of the paper is organized as follows.
Section~\ref{sec:setup} introduces the market, endowment, objective, and terminal
CVaR constraint. Section~\ref{sec:reformulation} develops the primal and dual
formulations. Section~\ref{sec:algorithm} presents the nested search algorithm
and its convergence analysis. Section~\ref{sec:numerics} studies the complete-
and incomplete-market benchmarks, and Section~\ref{sec:extension} adds
quadratic trading rate regularization and nonlinear price impact.
Section~\ref{sec:conclusion} concludes.

\section{Market Model and CVaR-Constrained Portfolio Problem}
\label{sec:setup}

This section formulates the dynamic portfolio management problem studied in the paper. We first introduce a general continuous-time market with traded assets, cumulative endowment, and potentially nontraded background risk. Portfolio strategies are evaluated through a running and terminal cost criterion and are required to satisfy a CVaR constraint on terminal losses. This formulation accommodates incomplete markets and provides the setting for the reformulation and algorithm developed in the subsequent sections.

We then specialize the general model to a one-stock linear--quadratic Black--Scholes benchmark with terminal CVaR constraints. The benchmark is analytically transparent: in the absence of a binding CVaR constraint, the optimal risky-asset holding reduces to the classical Merton dollar exposure, while an active constraint changes the feedback strategy to protect the lower tail of terminal wealth. This structure also provides a common framework for the numerical experiments with endowment risk and trading frictions.

We work on a filtered probability space
$(\Omega,\F,\mathds{F} = \{\F_t\}_{t\in[0,T]},\P)$
satisfying the usual conditions, where $[0,T]$ denotes the investment horizon. Let
$\left\{B_t=(B_t^{(1)},B_t^{(2)},\ldots,B_t^{(d)})^\top\right\}_{t\in[0,T]}$,
be a standard $d$-dimensional Brownian motion adapted to $\mathds{F}$, and let $\{B_t^\perp\}_{t\in[0,T]}$ be a 1-dim Brownian motion on $\mathds{F}$ that is independent of $\{B_t\}_{t\in[0,T]}$.

The market contains one risk-free asset and $m\leq d$ risky assets. The risk-free asset earns a constant interest rate $r\geq0$. The risky asset price vector
$S_t=(S_t^{(1)},S_t^{(2)},\ldots,S_t^{(m)})^\top$
satisfies
\begin{align}\label{dyn:stock}
dS_t = rS_t dt + \diag\{S_t\}\left(\mu(t,S_t)dt+\sigma(t,S_t)\,dB_t\right)\,, \qquad S_0\in(0,\infty)^m .
\end{align}
Here the $\mu:[0,T]\times\R^m\to\R^m$ is the vector of excess dollar returns, and $\sigma:[0,T]\times\R^m\to\R^{m\times d}$ is the volatility matrix. Since $m$ may be smaller than $d$, the risky assets need not span all Brownian shocks.

The investor receives a cumulative endowment $\zeta$ with dynamics
\begin{align}\label{dyn:endowment}
d\zeta_t = db_t+\beta_t\,dB_t + \beta_t^\perp dB_t^\perp.
\end{align}
The finite-variation process $b$ captures deterministic or predictable cash flows, including lump-sum payments, while the predictable row vector $\beta_t\in\R^{1\times d}$ captures the endowment's exposure to Brownian shocks. Only the component of $\beta_t$ lying in the span of the traded volatility matrix can be hedged by the risky assets. In contrast, 
the remaining component $\beta^\perp_t\in\R$, driven by the independent Brownian motion $B^\perp$, is nontraded background risk and is a source of market incompleteness in the model.

We assume that the filtered probability space supports a process $S$
satisfying~\eqref{dyn:stock}, and we fix this weak solution throughout
the paper. We impose the following regularity conditions on the exogenous coefficients and endowment processes: 
\begin{assumption}[Regularity of exogenous processes]\label{ass:market}
The functions $\mu$ and $\sigma$ are jointly  
continuous. In addition, $\mu$ and $\sigma$ are uniformly bounded, and for every $(t,s)\in[0,T]\times\mathbb R^m$,
$\sigma(t,s)\sigma(t,s)^\top$
is positive definite.
The process $b$ is predictable and of finite variation, with $b_0=0$,
and $\beta$ and $\beta^\perp$ are predictable. Moreover, there exists
$p>1$ such that  
\begin{align*}
\E\left[
|b|_T^{2p}
+
\left(\int_0^T \|\beta_t\|^2 + (\beta^\perp_t)^2\,dt\right)^{p}
\right]<\infty.
\end{align*}
\end{assumption}

Let $\phi = \left\{\phi_t = \left(\phi_t^{(1)}, \phi_t^{(2)}, \ldots, \phi_t^{(m)}\right)^\top\right\}_{t\in[0,T]}$ denote the portfolio of the investor, where
$\phi_t^{(i)}$ represents the number of shares held in the $(i)$--th risky asset. 
Available to the investor are trading strategies that are $\R^m$-valued predictable processes $\varphi=\left\{\varphi_t := \diag\{S_t\} \phi_t \right\}_{t\in[0,T]}$, 
where $\varphi_t^{(i)} = \phi_t^{(i)} S_t^{(i)}$ denotes the dollar risky exposure in the $(i)$--th risky asset. The remaining wealth  is invested in the risk-free asset, with the controlled wealth process $W^\varphi = \{W^\varphi_t\}_{t\in[0,T]}$ follows
\begin{align}\label{dyn:wealth}
dW_t^\varphi
  &=r\left(W_t^\varphi - \phi_t^\top S_t\right)dt + \phi^\top_t dS_t + d\zeta_t
\notag\\
&=\left(r W_t^\varphi+\varphi_t^\top \mu(t,S_t)\right)\,dt + db_t
+
\left(\varphi_t^\top\sigma(t,S_t)+\beta_t\right)\,dB_t + \beta^\perp_t dB_t^\perp,
\qquad W_0^\varphi=w_0\,,
\end{align}
where $w_0>0$ denotes the initial wealth of the investor.

Next, we denote by $\A$ the set of admissible strategies as follows:

\begin{assumption}\label{ass:action_set}
The set of admissible strategies is 
\[
\mathcal A
:=
\left\{
\varphi\in
L^2_{\mathbb F}([0,T]\times\Omega;\mathbb R^m):
\varphi_t\in A,\quad
dt\otimes d\mathbb P\text{-a.e.}
\right\}.
\]
where the action set $A\subset\mathbb R^m$ is nonempty, compact, and convex. 
\end{assumption}

The investor evaluates a strategy through the cost functional
\begin{align}\label{target:J}
J(\varphi)
:=
\E\left[
\int_0^T f(t,W_t^\varphi,\varphi_t)\,dt
+
g(W_T^\varphi)
\right],
\end{align}
where
$f 
$ is the running cost function and
$g$ is the terminal cost function. As the investor aims to minimize the cost functional \eqref{target:J}, the utility maximization framework fits within our setting by treating costs as negative rewards.

In addition to the cost functional, the investor requires a risk constraint on a function $\ell:\R\to\R$ applied to the terminal portfolio wealth. The risk constraint is given by the Conditional Value-at-Risk (CVaR) \citep{acerbi2002portfolio}. 
The CVaR at a confidence level $\alpha\in(0,1)$ for the loss random variable $\ell(W_T)$ is defined as 
\begin{align}\label{def:CVaR}
\operatorname{CVaR}_{\alpha}(\ell(W_T)) = \frac{1}{1-\alpha}\int_{\alpha}^1 \operatorname{VaR}_{u}(\ell(W_T)) du,
\end{align}
where the Value-at-Risk at a confidence level of $\alpha\in(0,1)$ is given by
\begin{align}\label{def:VaR}
\operatorname{VaR}_{\alpha}(\ell(W_T)) = \inf\{x\in\R: \P(\ell(W_T)\leq x)\geq \alpha\}. 
\end{align}
Given a confidence level $\alpha\in(0,1)$ and a risk limit $c\in\mathbb{R}$, the investor requires a strategy that satisfies
\begin{align}\label{constraint:cvar}
\operatorname{CVaR}_{\alpha}\bigl(\ell(W_T^\varphi)\bigr)\leq c .
\end{align}
For the choice $\ell(w) = -w$, the CVaR constraint limits the average severity of the worst $(1-\alpha)$ fraction of terminal losses to $c$.
Summarizing, the investor's CVaR-constrained portfolio problem is therefore
\begin{align}\label{prob:portfolio_cvar}
\inf_{\varphi\in\A}\quad
& J(\varphi) \\
\text{subject to}\quad
& \operatorname{CVaR}_{\alpha}\bigl(\ell(W_T^\varphi)\bigr)\leq c, \nonumber\\[0.5em]
& (S_t,W_t^\varphi) \text{ satisfy } \eqref{dyn:stock}\text{ and }\eqref{dyn:wealth}. \nonumber
\end{align}

Throughout the paper, we impose the following regularity conditions on
the objective and terminal loss functions.

\begin{assumption}[Regularity of cost functionals]
\label{ass:regularity}
The functions $f,g,\ell$ satisfy the following conditions.
\begin{enumerate}
    \item \label{ass:f_regularity} The function
$f:[0,T]\times\mathbb R\times A
\rightarrow
(-\infty,+\infty]$
is jointly Borel measurable and, for almost every $t\in[0,T]$, the
mapping $(w,a)\mapsto f(t,w,a)$ is proper, convex, and lower
semicontinuous. Moreover, there exists
$\underline f\in L^1(0,T)$ such that
$f(t,w,a)\ge \underline f(t)$
for almost every $t$ and all $(w,a)\in\mathbb R\times A$.
\item \label{ass:g_regularity}  The function
$g:\mathbb R\rightarrow(-\infty,+\infty]$
is proper, convex, and lower semicontinuous.
\item \label{ass:l_regularity} The terminal loss
$\ell:\mathbb R\rightarrow\mathbb R$
is convex and lower semicontinuous. Moreover, for the same $p>1$ as in
Assumption~\ref{ass:market}, there exists $L_\ell>0$ such that
\[
|\ell(w)|
\le
L_\ell(1+|w|^p),
\qquad w\in\mathbb R.
\]
\end{enumerate}

\end{assumption}

\begin{remark}\label{rmk:loss_integrable}
Together with Assumptions~\ref{ass:market} and~\ref{ass:action_set}, Assumption~\ref{ass:regularity}\ref{ass:l_regularity} ensures that
$\ell(W_T^\varphi)$ is integrable for every
$\varphi\in\mathcal A$. 
Indeed, the standard moment estimate for the wealth equation, together
with the boundedness of $\mu$ and $\sigma$ and the compactness of $A$,
gives
$\mathbb E\bigl[|W_T^\varphi|^{2p}\bigr]<\infty$
for every $\varphi\in\mathcal A$. Assumption~\ref{ass:regularity}~\ref{ass:l_regularity} then implies
$\mathbb E\bigl[|\ell(W_T^\varphi)|^2\bigr]<\infty$,
and hence $\ell(W_T^\varphi)$ is integrable. 
A uniform square-integrability estimate is
established in Proposition~\ref{prop:eta_compactification}.
 
\end{remark}

\begin{remark}\label{rmk:J_lower_bound}
Assumption~\ref{ass:regularity}\ref{ass:g_regularity} ensures that there exist $a_g,b_g\in\mathbb R$ such that
$g(w)\ge a_gw+b_g$,
for all $w\in\mathbb R$.
Together with Assumption~\ref{ass:regularity}\ref{ass:f_regularity} on $f$ and the uniform
first-moment bound on admissible wealth processes, this yields
$\inf_{\varphi\in\mathcal A}J(\varphi)>-\infty$. We henceforth fix $J_{\rm low}\in\mathbb R$ such that
$J(\varphi)\ge J_{\rm low}$ for
any $\varphi\in\mathcal A.$

Allowing the objective costs to take the value $+\infty$ permits
domain restrictions in utility-based criteria. In particular, the
framework accommodates CARA terminal costs, as well as extended-valued
CRRA and logarithmic costs whenever a finite-cost feasible strategy
exists.
\end{remark}

For the existence of a primal optimizer, we require the feasibility of the optimization problem.
\begin{assumption}[Feasibility]\label{ass:feasibility} 
There exists at least one admissible strategy $\varphi\in\mathcal A$ such that 
\[ J(\varphi)<\infty\qquad\text{and}\qquad\operatorname{CVaR}_{\alpha}(\ell(W_T^\varphi))\le c. \] 
\end{assumption}

A sufficient condition for the feasibility assumption is that $0\in A$, that is the no-trade wealth
\begin{align*}
W_T^0=e^{rT}w_0+\int_0^T e^{r(T-t)}\,d\zeta_t    
\end{align*}
satisfies $\operatorname{CVaR}_{\alpha}(\ell(W_T^0))\leq c$, and the corresponding objective $J(0)$ is finite.

The formulation in Section~\ref{sec:setup} is deliberately stated at a general level. It accommodates alternative objective specifications, including CARA
terminal costs and, whenever the effective-domain and finite-feasibility
requirements are satisfied, CRRA and logarithmic costs, as well as richer market dynamics and additional state variables to include consumption control or market frictions. The reformulation and algorithm developed in Sections~\ref{sec:reformulation} and~\ref{sec:algorithm} are driven primarily by the convex structure of the CVaR representation and do not depend on the particular objective function or market specification.

\section{Constrained Convex Optimization Reformulation}\label{sec:reformulation}
In this section, we reformulate the risk-constrained stochastic control problem \eqref{prob:portfolio_cvar} as a convex constrained optimization problem. The main device is the auxiliary-variable representation of CVaR, which converts the terminal risk constraint into a convex constraint in both the control and an additional scalar threshold variable. This formulation then allows us to introduce a Lagrangian dual problem and establish strong duality under an additional Slater-type condition.

\subsection{Reformulation and Primal Problem}\label{ssec:primal}
By Remark~2.4, $\ell(W_T^\varphi)$ is integrable for every
$\varphi\in\mathcal A$. Hence, its CVaR at level
$\alpha\in(0,1)$ admits the representation introduced in
\cite{rockafellar2000optimization}:
\begin{align}\label{cvar:rockafellar}
\operatorname{CVaR}_{\alpha}(\ell(W_T^\varphi)) = \inf_{\eta\in\mathbb R} \left\{ \eta+\frac{1}{1-\alpha}\mathbb E\big[(\ell(W_T^\varphi)-\eta)^+\big] \right\}.    
\end{align}
The infimum is attained, and its minimizers are the generalized
$\alpha$-quantiles of $\ell(W_T^\varphi)$. 
Further, for each $(\varphi,\eta)\in\mathcal A\times\mathbb R$, define 
\begin{equation}\label{eq:constraint_C}
C(\varphi,\eta) := \eta+\frac{1}{1-\alpha} \mathbb E\big[(\ell(W_T^\varphi)-\eta)^+\big]-c . \end{equation}
Then the constraint $\operatorname{CVaR}_{\alpha}(\ell(W_T^\varphi))\le c$ is equivalent to the existence of $\eta\in\mathbb R$ such that $C(\varphi,\eta)\le 0$. 
Thus the portfolio optimization problem can be written as 
\begin{align*}
p^\star := \inf_{\varphi\in\mathcal A,\ \eta\in\mathbb R} J(\varphi) \quad \text{subject to} \quad C(\varphi,\eta)\le 0 . \tag{P} \label{eq:primal_problem} 
\end{align*}
Note that the scalar variable $\eta$ is not a trading decision. It is an auxiliary threshold variable introduced by the CVaR representation. This formulation separates the portfolio control $\varphi$ from the scalar risk threshold $\eta$, while preserving convexity.

We now present the main analytical properties of the primal formulation. These properties will be used later to derive the Lagrangian dual problem and the dual-based numerical algorithm.

\begin{proposition}[Convexity of the primal problem]\label{prop:convexity}
Under Assumptions \ref{ass:action_set} and \ref{ass:regularity}, the extended-valued functional $J:\mathcal A\to(-\infty,+\infty] $ defined in \eqref{target:J} is convex in $\varphi$, and the finite-valued mapping $C:\mathcal A\times\mathbb R\to\mathbb R$ defined in \eqref{eq:constraint_C} is jointly convex in $(\varphi,\eta)$, and the primal problem is a convex optimization problem.
\end{proposition}
\begin{proof}
For any $\varphi^1,\varphi^2\in\mathcal A$ and $\theta\in[0,1]$, define $\varphi^\theta:=\theta\varphi^1+(1-\theta)\varphi^2$. By the linearity of the state equation and the uniqueness of solutions,
\[
W_t^{\varphi^\theta}=\theta W_t^{\varphi^1}+(1-\theta)W_t^{\varphi^2},\qquad 0\le t\le T.
\]
Since $f(t,\cdot,\cdot)$ is jointly convex and $g$ is convex, we have
\[
f(t,W_t^{\varphi^\theta},\varphi_t^\theta)\le\theta f(t,W_t^{\varphi^1},\varphi_t^1)+(1-\theta)f(t,W_t^{\varphi^2},\varphi_t^2),
\]
and similarly $g(W_T^{\varphi^\theta})\le \theta g(W_T^{\varphi^1})+(1-\theta)g(W_T^{\varphi^2})$. Taking expectations gives $J(\varphi^\theta)\le \theta J(\varphi^1)+(1-\theta)J(\varphi^2)$. Therefore, the objective function $J(\varphi)$ is convex in the control $\varphi$ on $\mathcal A$.

It remains to show that $C$ is convex. Let $\eta^\theta:=\theta\eta^1+(1-\theta)\eta^2$. Since $\ell$ is convex and $W_T^{\varphi}$ is affine in $\varphi$,
\[
\ell(W_T^{\varphi^\theta})\le\theta \ell(W_T^{\varphi^1})+(1-\theta)\ell(W_T^{\varphi^2}).
\]
Therefore,
\[
\ell(W_T^{\varphi^\theta})-\eta^\theta\le\theta(\ell(W_T^{\varphi^1})-\eta^1)+(1-\theta)(\ell(W_T^{\varphi^2})-\eta^2).
\]
Because $x\mapsto x_+$ is convex and nondecreasing, it follows that 
\[
(\ell(W_T^{\varphi^\theta})-\eta^\theta)^+ \le \theta(\ell(W_T^{\varphi^1})-\eta^1)^+ + (1-\theta)(\ell(W_T^{\varphi^2})-\eta^2)^+.
\]
Taking expectations and adding the affine term $\eta$ shows that $C(\varphi,\eta)$ is jointly convex. Hence the feasible set $\{(\varphi,\eta):\varphi\in\mathcal A,\ C(\varphi,\eta)\le 0\}$ is convex, and the primal problem is convex.
\end{proof}

We next establish a uniform square-integrability estimate for the
terminal losses and use it to restrict the auxiliary threshold to a
common compact interval that does not depend on the constraint level $c$.

\begin{proposition}[Compactification of the CVaR threshold]
\label{prop:eta_compactification}
Under Assumptions~\ref{ass:market}, \ref{ass:action_set}, and
\ref{ass:regularity},  
there exists
$M_X<\infty$ such that
$\sup_{\varphi\in\mathcal A}
\E\left[|\ell(W_T^\varphi)|^2\right]
\le M_X^2 $.
Let 
$E=
\left[
-\frac{M_X}{\sqrt{\alpha}},
\frac{M_X}{\sqrt{1-\alpha}}
\right]$,
then for every $\varphi\in\mathcal A$,
\[
\inf_{\eta\in\mathbb R}
\left\{
\eta+\frac{1}{1-\alpha}
\E[(\ell(W_T^\varphi)-\eta)^+]
\right\}
=
\inf_{\eta\in E}
\left\{
\eta+\frac{1}{1-\alpha}
\E[(\ell(W_T^\varphi)-\eta)^+]
\right\}.
\]
Consequently, \eqref{eq:primal_problem} is equivalent to
\[
p^\star
=
\inf_{\varphi\in\mathcal A,\ \eta\in E}
J(\varphi)
\quad
\textnormal{subject to}
\quad
C(\varphi,\eta)\le 0.
\tag{$\mathrm P_E$}
\label{eq:compact_primal_problem}
\]
\end{proposition}

\begin{proof}
By the standard moment estimate for SDEs with coefficients of linear
growth, see \cite{mao2007stochastic} or
\cite[Theorem 2.1]{kim2014p}, for  $p\ge 1$ from Assumption~\ref{ass:market}, there exists
$C_p>0$, independent of $\varphi\in\mathcal A$, such that
\[
\sup_{\varphi\in\mathcal A}
\mathbb E[|W_T^\varphi|^{2p}]
\le C_p(1+|w_0|^{2p})<\infty.
\]
Since $|\ell(w)|\le K(1+|w|^p)$, we have
$|\ell(W_T^\varphi)|^2
\le
2K^2\left(1+|W_T^\varphi|^{2p}\right)$,
and therefore
$\sup_{\varphi\in\mathcal A}
\mathbb E[|\ell(W_T^\varphi)|^2]< \infty$. 
Hence, there exists 
$M_X>0$ that only depends on 
the compact set $A$ and exogenous market coefficients, 
such that
$\sup_{\varphi\in\mathcal A}
\mathbb E[|\ell(W_T^\varphi)|^2]\le M_X^2$ .

Now fix $\varphi\in\mathcal A$ and write
$
X=\ell(W_T^\varphi).$
The minimizers of
$
\eta\mapsto
\eta+\frac{1}{1-\alpha}\E[(X-\eta)^+]
$
are the $\alpha$-quantiles of $X$, namely the points $\eta$ satisfying
\[
\mathbb P(X<\eta)\le \alpha\le \mathbb P(X\le \eta).
\]
If
$\eta<-\frac{M_X}{\sqrt{\alpha}},
$
then
\[
\mathbb P(X\le \eta)
\le
\mathbb P(|X|\ge |\eta|)
\le
\frac{\E[|X|^2]}{\eta^2}
<
\alpha.
\]
Thus $\eta$ cannot be an $\alpha$-quantile. Similarly, if
$
\eta>\frac{M_X}{\sqrt{1-\alpha}},
$
then
\[
\mathbb P(X\ge \eta)
\le
\mathbb P(|X|\ge \eta)
\le
\frac{\E[|X|^2]}{\eta^2}
<
1-\alpha.
\]
Hence $\mathbb P(X<\eta)>\alpha$, so $\eta$ cannot be an
$\alpha$-quantile. Therefore every minimizer lies in $E$, and the infimum
over $\mathbb R$ equals the infimum over $E$.

The equivalence between \eqref{eq:primal_problem} and
\eqref{eq:compact_primal_problem} follows directly.
\end{proof}

\begin{proposition}[Existence of a primal optimizer]
\label{prop:primal_existence}
Assume Assumptions~\ref{ass:market}, \ref{ass:action_set},
\ref{ass:regularity}, and~\ref{ass:feasibility} hold. Then the compactified
primal problem
\[
p^\star
=
\inf_{\varphi\in\mathcal A,\ \eta\in E}
J(\varphi)
\quad
\textnormal{subject to}
\quad
C(\varphi,\eta)\le 0
\]
admits an optimal solution
$
(\varphi^\star,\eta^\star)\in\mathcal A\times E$.
Consequently, the original primal problem over
$\mathcal A\times\mathbb R$ also admits an optimal solution. In particular,
one may choose an optimal solution with $\eta^\star\in E$.
\end{proposition}

\begin{proof}

By Assumption~\ref{ass:feasibility} and
Proposition~\ref{prop:eta_compactification}, the feasible set of the
compactified problem is nonempty.  Define
\[
\mathbb A
:=
\{(\varphi,\eta)\in\mathcal A\times E:C(\varphi,\eta)\le 0\}.
\]

Since $A$ is compact,
the set $\mathcal A$ is bounded in
$L^2_{\mathbb F}([0,T]\times\Omega;\mathbb R^m)$. Moreover, $\mathcal A$
is convex by the convexity of $A$, and it is closed: if
$\varphi^n\to\varphi$ in $L^2_{\mathbb F}$, then, up to a subsequence,
$\varphi_t^n\to\varphi_t$ $dt\otimes d\mathbb P$-a.e.; since $A$ is
closed, $\varphi_t\in A$ a.e. Hence $\varphi\in\mathcal A$. Thus
$\mathcal A$ is closed, bounded, and convex in the Hilbert space
$L^2_{\mathbb F}$, and is therefore weakly compact. Since $E$ is compact,
$\mathcal A\times E$ is weakly compact.

It remains to show that $\mathbb A$ is weakly closed. By
Proposition~\ref{prop:convexity}, $C$ is convex. Moreover, 
by the linearity of the controlled wealth dynamics~\eqref{dyn:wealth}, we can see that for $\varphi^1, \varphi^2 \in\A$, 
$$
de^{-rt}(W_t^{\varphi^1} - W_t^{\varphi^2}) =  e^{-rt} \left(\varphi_t^1 - \varphi_t^2\right)^\top \left(\mu(t, S_t) dt + \sigma(t, S_t) dB_t\right),
$$
which implies
\begin{align}
\E\left[(W_t^{\varphi^1} - W_t^{\varphi^2})^2\right]
&\leq 2\E \left[\int_0^t  e^{2r(t-u)}\left( t\left|\left(\varphi_u^1 - \varphi_u^2\right)^\top \mu(u, S_u)\right|^2 + \left\|\left(\varphi_u^1 - \varphi_u^2\right)^\top \sigma(u, S_u) \right\|^2 \right)du \right]\notag\\ 
&\leq M\E\left[\int_0^t \left\|\varphi_u^1 - \varphi_u^2\right\|^2 \ du\right],\label{eq:stability}
\end{align}
where $M$ only depends on  $r$,  $T$, the set $A$, and the uniform bounds of $\mu$ and $\sigma$. 
The above stability estimate~\eqref{eq:stability} for the wealth equation, together with the lower
semicontinuity of $\ell$ and Fatou's lemma, implies that $C$ is lower
semicontinuous in the strong topology of
$L^2_{\mathbb F}\times\mathbb R$.
Since a convex strongly lower
semicontinuous functional is weakly lower semicontinuous, $C$ is weakly lower
semicontinuous. Therefore the sublevel set of $C$
\[
\mathbb A=\{(\varphi,\eta)\in\mathcal A\times E:C(\varphi,\eta)\le0\}
\]
is weakly closed. Hence $\mathbb A \subset \mathcal A\times E$ is weakly compact.

Let $(\varphi^n,\eta^n)\subset\mathbb A$ be a minimizing sequence, so that
$
J(\varphi^n)\to p^\star .
$
By weak compactness of $\mathbb A$, there exist a subsequence, still denoted
by $(\varphi^n,\eta^n)$, and a pair
$(\varphi^\star,\eta^\star)\in\mathbb A$ such that
\[
\varphi^n\rightharpoonup\varphi^\star
\quad\text{in }L^2_{\mathbb F},
\qquad
\eta^n\to\eta^\star .
\]

We next verify that $J$ is weakly lower semicontinuous. Suppose first
that $\varphi^n\to\varphi$ strongly in $L^2_{\mathbb F}$. By
\eqref{eq:stability},
\[
W^{\varphi^n}\to W^\varphi
\quad\text{in }L^2([0,T]\times\Omega),
\qquad
W_T^{\varphi^n}\to W_T^\varphi
\quad\text{in }L^2(\Omega).
\]
Passing to a subsequence, the corresponding convergences hold almost
everywhere.

Since
$f(t,w,a)\ge \underline f(t)$ with $\underline f\in L^1(0,T)$,
the function $f-\underline f$ is nonnegative. The lower
semicontinuity of $f$ and Fatou's lemma therefore yield
\[
\mathbb E\left[
\int_0^T
f(t,W_t^\varphi,\varphi_t)\,dt
\right]
\le
\liminf_{n\to\infty}
\mathbb E\left[
\int_0^T
f(t,W_t^{\varphi^n},\varphi_t^n)\,dt
\right].
\]

Moreover, since $g$ is proper, convex, and lower semicontinuous, there exist $a_g,b_g\in\mathbb R$ such that
$g(w)\ge a_gw+b_g$,
for all $w\in\mathbb R$.
Applying Fatou's lemma to the nonnegative function
$g(w)-a_gw-b_g$, and using the $L^1$ convergence of
$W_T^{\varphi^n}$, gives
\[
\mathbb E[g(W_T^\varphi)]
\le
\liminf_{n\to\infty}
\mathbb E[g(W_T^{\varphi^n})].
\]
Thus $J$ is strongly lower semicontinuous.

Since $J$ is convex by Proposition~\ref{prop:convexity},
it is weakly lower semicontinuous. Hence
\[
J(\varphi^\star)
\le
\liminf_{n\to\infty}J(\varphi^n)
=
p^\star .
\]
Since $(\varphi^\star,\eta^\star)\in\mathbb A$ is feasible, the reverse
inequality $p^\star\le J(\varphi^\star)$ follows from the definition of
$p^\star$. Therefore
$
J(\varphi^\star)=p^\star,
$
and $(\varphi^\star,\eta^\star)$ is a primal optimizer.

Finally, Proposition~\ref{prop:eta_compactification} shows that restricting
$\eta$ to $E$ does not change the value of the primal problem. Hence the
same pair is also an optimal solution of the original primal problem over
$\mathcal A\times\mathbb R$.
\end{proof}

Finally, we establish the uniqueness of the optimal control, under strictly convexity assumption.

\begin{assumption}\label{ass:strict_convexity}
The objective functional $J$ is strictly convex on $\mathcal A$; that is, 
for any distinct $\varphi,\psi\in\mathcal A$ satisfying
$J(\varphi)<\infty$ and $J(\psi)<\infty$,
and any $\theta\in(0,1)$,
\[
J\bigl(\theta\varphi+(1-\theta)\psi\bigr)
<
\theta J(\varphi)+(1-\theta)J(\psi).
\]
\end{assumption}

In the following, we provide some explicit conditions on the problem parameters that can guarantee the above the strict convexity of $J$.

\begin{proposition}\label{prop:strict_convexity}
Suppose that Assumptions~\ref{ass:market},
\ref{ass:action_set}, and \ref{ass:regularity} hold.
Then the objective functional $J$ is strictly convex on $\mathcal A$, if one of the following conditions hold:
\begin{enumerate}
    \item \label{ass:strict_convex_f} For almost every $t\in[0,T]$, the function
$(w,a)\mapsto f(t,w,a)$ is strictly convex on its effective domain, 
in the sense that, for every $\theta\in(0,1)$,
\[
\begin{aligned}
&f\bigl(
t,\theta w+(1-\theta)w',
\theta a+(1-\theta)a'
\bigr)<
\theta f(t,w,a)
+(1-\theta)f(t,w',a')
\end{aligned}
\]
whenever $(w,a)\neq (w',a')$ and both
$f(t,w,a)$ and $f(t,w',a')$ are finite.

\item \label{ass:strict_convex_g} 
There exists $\underline\nu>0$ such that
$\sigma(t,s)\sigma(t,s)^\top
\succeq
\underline\nu I_m$, for all $(t,s)\in[0,T]\times\mathbb R^m$, and the function $g$ is strictly convex on its effective domain. More precisely, for every $\theta\in(0,1)$,
\[
g\bigl(\theta w+(1-\theta)w'\bigr)
<
\theta g(w)+(1-\theta)g(w')
\]
whenever $w\neq w'$ and both
$g(w)$ and $g(w')$ are finite.
\end{enumerate}
\end{proposition}

Condition~\ref{ass:strict_convex_g} covers several standard expected-utility
objectives. In particular, when \(f\equiv0\), it applies to the negative
exponential (CARA) utility
\[
g(w)=e^{-\gamma w}, \qquad \gamma>0,
\]
as well as to the lower-semicontinuous extended-valued negative CRRA
cost.
For $\gamma>0$, define
\[
U_\gamma(w)=
\begin{cases}
\dfrac{w^{1-\gamma}-1}{1-\gamma}, & \gamma\neq1,\\[1ex]
\log w, & \gamma=1,
\end{cases}
\qquad w>0,
\]
and define the negative CRRA cost $g_\gamma$ as the
lower-semicontinuous extension of $-U_\gamma$ to $\mathbb R$, namely
\[
g_\gamma(w)
=
\begin{cases}
-U_\gamma(w), & w>0,\\
\lim_{x\downarrow0}-U_\gamma(x), & w=0,\\
+\infty, & w<0.
\end{cases}
\]
For every $\gamma>0$, $g_\gamma$ is proper, convex, lower
semicontinuous, and strictly convex on its effective domain.
The limiting case at \(\gamma=1\) is the negative logarithmic utility
\(g(w)=-\log w\).

\begin{proof}
Let $\varphi,\psi\in\mathcal A$ be distinct in $L^2_{\mathbb F}$ and satisfy
$J(\varphi),J(\psi)<\infty$. Let
$\theta\in(0,1)$, and define
$\varphi^\theta
:=
\theta\varphi+(1-\theta)\psi$.
Since $\mathcal A$ is convex, $\varphi^\theta\in\mathcal A$. Moreover, the
wealth equation is affine in the control, so
\[
W^{\varphi^\theta}
=
\theta W^\varphi+(1-\theta)W^\psi.
\]

Since $\varphi\neq\psi$ in $L^2_{\mathbb F}$, the set
$\left\{
(t,\omega):
\varphi_t(\omega)\neq\psi_t(\omega)
\right\}$
has positive $dt\otimes d\mathbb P$ measure. On this set,
$(W_t^\varphi,\varphi_t)
\neq
(W_t^\psi,\psi_t)$.
If Condition \eqref{ass:strict_convex_f} holds, then  
\begin{align*}
&\mathbb E\int_0^T
f\bigl(t,W_t^{\varphi^\theta},\varphi_t^\theta\bigr)\,dt<
\theta\mathbb E\int_0^T
f(t,W_t^\varphi,\varphi_t)\,dt
+
(1-\theta)\mathbb E\int_0^T
f(t,W_t^\psi,\psi_t)\,dt.
\end{align*}
By the convexity of $g$,
\[
\mathbb E\bigl[g(W_T^{\varphi^\theta})\bigr]
\le
\theta\mathbb E\bigl[g(W_T^\varphi)\bigr]
+
(1-\theta)\mathbb E\bigl[g(W_T^\psi)\bigr].
\]
Combining the two inequalities gives
$J(\varphi^\theta)
<
\theta J(\varphi)+(1-\theta)J(\psi)$.

If Condition \ref{ass:strict_convex_g} holds, 
let
$\varphi,\psi\in\mathcal A$ be distinct in
$L^2_{\mathbb F}$ and satisfy
$J(\varphi),J(\psi)<\infty$. Define
$u_t:=\varphi_t-\psi_t$,
$X_t:=W_t^\varphi-W_t^\psi$.
Since the endowment terms cancel, the wealth difference satisfies
\[
dX_t
=
\left(rX_t+u_t^\top\mu(t,S_t)\right)dt
+
u_t^\top\sigma(t,S_t)\,dB_t,
\qquad X_0=0.
\]
Let
$\overline\mu
:=
\sup_{(t,s)\in[0,T]\times\mathbb R^m}
|\mu(t,s)|<\infty$.
Applying It\^o's formula to $|X_t|^2$ and taking expectations gives
\[
\frac{d}{dt}\mathbb E[|X_t|^2]
=
2r\mathbb E[|X_t|^2]
+
2\mathbb E\!\left[
X_tu_t^\top\mu(t,S_t)
\right]
+
\mathbb E\!\left[
u_t^\top
\sigma(t,S_t)\sigma(t,S_t)^\top
u_t
\right].
\]
By Young's inequality,
\[
2X_tu_t^\top\mu(t,S_t)
\ge
-\frac{\underline\nu}{2}|u_t|^2
-
\frac{2\overline\mu^2}{\underline\nu}|X_t|^2.
\]
Together with
$u_t^\top
\sigma(t,S_t)\sigma(t,S_t)^\top
u_t
\ge
\underline\nu |u_t|^2$,
and $r\ge0$, this yields
\[
\frac{d}{dt}\mathbb E[|X_t|^2]
\ge
-K\mathbb E[|X_t|^2]
+
\frac{\underline\nu}{2}\mathbb E[|u_t|^2],
\qquad
K:=\frac{2\overline\mu^2}{\underline\nu}.
\]
Since $X_0=0$, multiplying by $e^{Kt}$ and integrating gives
\begin{equation}\label{eq:lower_bound_wealth}
    \mathbb E\!\left[
|W_T^\varphi-W_T^\psi|^2
\right]
\ge
\frac{\underline\nu}{2}
\int_0^T
e^{-K(T-t)}
\mathbb E[|\varphi_t-\psi_t|^2]\,dt \ge
\frac{\underline\nu}{2}e^{-KT}
\|\varphi-\psi\|_{L^2_{\mathbb F}}^2.
\end{equation}
Since $\varphi\neq\psi$ in $L^2_{\mathbb F}$, it follows that
$\mathbb P\!\left(
W_T^\varphi\neq W_T^\psi
\right)>0$.
Hence, by the strict convexity of $g$,
\[
\mathbb E\bigl[g(W_T^{\varphi^\theta})\bigr]
<
\theta\mathbb E\bigl[g(W_T^\varphi)\bigr]
+
(1-\theta)\mathbb E\bigl[g(W_T^\psi)\bigr].
\]
By the convexity of $f$,
\begin{align*}
&\mathbb E\int_0^T
f\bigl(t,W_t^{\varphi^\theta},\varphi_t^\theta\bigr)\,dt\le 
\theta\mathbb E\int_0^T
f(t,W_t^\varphi,\varphi_t)\,dt
+
(1-\theta)\mathbb E\int_0^T
f(t,W_t^\psi,\psi_t)\,dt.
\end{align*}
Combining the two inequalities also implies 
$J(\varphi^\theta)
<
\theta J(\varphi)+(1-\theta)J(\psi)$.
Thus, $J$ is strictly convex on $\mathcal A$. 
\end{proof}

\begin{proposition}[Uniqueness of the optimal control]
\label{prop:unique_control}
Suppose that Assumptions~\ref{ass:market},
\ref{ass:action_set}, \ref{ass:regularity},
\ref{ass:feasibility}, and~\ref{ass:strict_convexity} hold.
Then 
the primal optimal control $\varphi^\star$ is unique.
\end{proposition}

\begin{proof}
Let $(\varphi^1,\eta^1)$ and $(\varphi^2,\eta^2)$ be two primal optimizers. For $\theta\in(0,1)$, define 
\[
\varphi^\theta:=\theta\varphi^1+(1-\theta)\varphi^2, \qquad \eta^\theta:=\theta\eta^1+(1-\theta)\eta^2.
\]
By convexity of the feasible set, $(\varphi^\theta,\eta^\theta)$ is feasible. If $\varphi^1\ne\varphi^2$, by strict convexity of $J$, 
\[
J(\varphi^\theta) < \theta J(\varphi^1)+(1-\theta)J(\varphi^2) = p^\star. 
\]
This contradicts the optimality of $(\varphi^1,\eta^1)$ and $(\varphi^2,\eta^2)$. Therefore $\varphi^1=\varphi^2$. 
\end{proof}

\begin{remark}[Uniqueness of the CVaR threshold]
The auxiliary variable $\eta^\star$ is not unique in general, since the
objective $J$ does not depend on $\eta$. For a fixed optimal control
$\varphi^\star$, the set of admissible thresholds is
\[
\left\{
\eta\in E:
H_{\varphi^\star}(\eta)\le c
\right\},
\qquad
H_{\varphi^\star}(\eta)
:=
\eta+\frac{1}{1-\alpha}
\E\big[(\ell(W_T^{\varphi^\star})-\eta)^+\big].
\]
Thus $\eta^\star$ is unique only if this set is a singleton. In particular,
if the CVaR constraint is active and $c=
\operatorname{CVaR}_\alpha(\ell(W_T^{\varphi^\star}))$, then
$\eta^\star$ must be a minimizer of $H_{\varphi^\star}$. In this case,
uniqueness of $\eta^\star$ is equivalent to uniqueness of the
$\alpha$-quantile of the terminal loss
\[
X^\star:=\ell(W_T^{\varphi^\star}).
\]
A sufficient condition is that the distribution function of $X^\star$ is
strictly increasing in a neighborhood of its $\alpha$-quantile. If the
constraint is slack, or if the quantile set contains an interval, then
$\eta^\star$ need not be unique.
\end{remark}

\subsection{Dual formulation}\label{ssec:dual}
We introduce the Lagrangian dual problem associated with the compactified primal problem \eqref{eq:compact_primal_problem}. 

For $(\varphi,\eta)\in\mathcal A\times E$ and $\lambda\ge 0$, define the Lagrangian 
\begin{align*} 
\mathcal L(\varphi,\eta,\lambda) :&= J(\varphi)+\lambda C(\varphi,\eta) 
\\
&=J(\varphi) + \lambda\left[ \eta+ \frac{1}{1-\alpha} \E\big[(\ell(W_T^\varphi)-\eta)^+\big] -c \right].
\end{align*}
The corresponding dual function is 
\begin{equation}\label{eq:dual_function}
    q(\lambda) := \inf_{\varphi\in\mathcal A,\ \eta\in E} \mathcal L(\varphi,\eta,\lambda), \qquad \lambda\ge 0, 
\end{equation}
and the dual problem is 
\[ 
d^\star := \sup_{\lambda\ge 0}q(\lambda). \tag{D} \label{eq:dual_problem} 
\]
More explicitly,
\[
d^*=\sup_{\lambda\ge 0}\inf_{\varphi\in\mathcal A,\ \eta\in E}\left\{
J(\varphi)+\lambda \left(\eta+\frac{1}{1-\alpha}\mathbb E\left[
\left(\ell(W_T^\varphi)-\eta\right)^+
\right]-c\right)
\right\}.
\]

\begin{remark} By Proposition~\ref{prop:eta_compactification}, restricting $\eta$ to $E$ does not change the primal problem. It also does not change the dual formulation. Indeed, for fixed $\varphi$ and $\lambda>0$, minimizing $\mathcal L(\varphi,\eta,\lambda)$ over $\eta$ is equivalent to minimizing 
\[ \eta+ \frac{1}{1-\alpha} \E\big[(\ell(W_T^\varphi)-\eta)^+\big] \] 
over $\eta$. Proposition~\ref{prop:eta_compactification} shows that this minimum is attained in $E$. When $\lambda=0$, the Lagrangian is independent of $\eta$. Hence 
\[ 
\inf_{\varphi\in\mathcal A,\ \eta\in\mathbb R} \mathcal L(\varphi,\eta,\lambda) = \inf_{\varphi\in\mathcal A,\ \eta\in E} \mathcal L(\varphi,\eta,\lambda), \qquad \lambda\ge0. 
\] 
\end{remark}

We now establish the property of the dual formulation \eqref{eq:dual_problem} and its connection with the primal formulation \eqref{eq:compact_primal_problem}.

\begin{proposition}[Existence of Lagrange minimizers]
\label{prop:lagrangian_minimizer_existence}
Under Assumptions~\ref{ass:market}, \ref{ass:action_set}, 
\ref{ass:regularity}, and \ref{ass:feasibility}, for every $\lambda\ge0$, the inner problem of \eqref{eq:dual_problem},
$
\inf_{\varphi\in\mathcal A,\ \eta\in E}
\mathcal L(\varphi,\eta,\lambda)
$,
admits a minimizer. 

Suppose, in addition, that Assumption~\ref{ass:strict_convexity} holds.
Then, for every $\lambda\ge0$, the Lagrangian is strictly convex in its
control component on its effective domain: for any
$(\varphi^1,\eta^1),(\varphi^2,\eta^2)\in\mathcal A\times E$ with
$\varphi^1\neq\varphi^2$ in $L^2_{\mathbb F}$ and $J(\varphi^1), J(\varphi^2)<\infty$, and any
$\theta\in(0,1)$,
\[
\begin{aligned}
&\mathcal L\bigl(
\theta\varphi^1+(1-\theta)\varphi^2,
\theta\eta^1+(1-\theta)\eta^2,
\lambda
\bigr)
<
\theta\mathcal L(\varphi^1,\eta^1,\lambda)
+
(1-\theta)\mathcal L(\varphi^2,\eta^2,\lambda).
\end{aligned}
\]
Consequently, all Lagrangian minimizers at a fixed $\lambda$ have the
same control component.
\end{proposition}

\begin{proof}
For fixed $\lambda\ge0$, the Lagrangian is given by
\[
\mathcal L(\varphi,\eta,\lambda)
=
J(\varphi)+\lambda C(\varphi,\eta).
\]
By the compactification result, the minimization is over
$\mathcal A\times E$. As in the proof of
Proposition~\ref{prop:primal_existence}, $\mathcal A$ is weakly compact and
$E$ is compact. Moreover, $J$ and $C$ are weakly lower semicontinuous, and
therefore $\mathcal L(\cdot,\cdot,\lambda)$ is weakly lower semicontinuous on
$\mathcal A\times E$. Hence the infimum is attained.

If Assumption \ref{ass:strict_convexity} holds, 
$J$ is strictly convex on
$\mathcal A$, while $C$ is jointly convex in $(\varphi,\eta)$. Hence, for
every $\lambda\ge0$, the Lagrangian
\[
\mathcal L(\varphi,\eta,\lambda)
=
J(\varphi)+\lambda C(\varphi,\eta)
\]
is strictly convex in its control component. Therefore, two Lagrangian
minimizers at the same multiplier cannot have different control
components, by an argument similar to the proof of Proposition \ref{prop:unique_control}.
\end{proof}

We next present a weak duality result, which states that the dual value is not larger than the primal value.
\begin{proposition}[Weak duality]\label{prop:weak_duality}
Under Assumptions~\ref{ass:market}, \ref{ass:action_set}, 
\ref{ass:regularity}, and \ref{ass:feasibility}, for every $\lambda\ge 0$, we have $q(\lambda)\le p^*$. Consequently,
$d^*\le p^*$.
\end{proposition}
\begin{proof}
Let $(\varphi,\eta)$ be primal feasible. Then $C(\varphi,\eta)\le 0$. Since $\lambda\ge 0$,
\[
\mathcal L(\varphi,\eta,\lambda)
=
J(\varphi)+\lambda C(\varphi,\eta)
\le
J(\varphi).
\]
Taking the infimum over feasible $(\varphi,\eta)$ gives $q(\lambda)\le p^*$. Taking the supremum over $\lambda\ge 0$ gives $d^*\le p^*$.
\end{proof}

To establish strong duality, we require the following Slater's condition. This condition strengthens Assumption \ref{ass:feasibility} to  strict feasibility, which is standard in the optimization literature on guarantee strong duality (see \textit{e.g.}, \cite{boyd2004convex}).

\begin{assumption}[Slater condition]
\label{ass:slater}
There exist $\bar\varphi\in\mathcal A$ and $\delta>0$ such that
\[
J(\bar\varphi)<\infty,
\qquad
\operatorname{CVaR}_{\alpha}
\bigl(\ell(W_T^{\bar\varphi})\bigr)
\le c-\delta.
\]
\end{assumption}

\begin{remark}\label{rmk:eta_bar}
    Under Assumption \ref{ass:slater},  Proposition \ref{prop:eta_compactification} ensures that choosing $\bar \eta$ that minimizes $C(\bar \varphi, \eta)$, guarantees that
    \[
    C(\bar \varphi, \bar \eta)=\operatorname{CVaR}_\alpha(\ell(W_T^{\bar\varphi}))-c\le -\delta.
    \]
\end{remark}

Now we are ready to state the strong duality result.

\begin{theorem}[Strong duality]\label{thm:strong_duality}
Under Assumptions~\ref{ass:market}, \ref{ass:action_set}, \ref{ass:regularity}, and~\ref{ass:slater},
the primal problem and dual problem have the same optimal value:
$p^*=d^*$.
Moreover, the dual problem admits an optimal multiplier $\lambda^*\ge 0$. 
\end{theorem}

\begin{proof}
Define the perturbation value function
\[
v(u):=
\inf_{\varphi\in\mathcal A,\ \eta\in E}\left\{
J(\varphi): C(\varphi,\eta)\le u
\right\}.
\]
Then $v(0)=p^*$. Since $J$ and $C$ are convex, $v$ is convex. Assumption \ref{ass:slater} implies $v(u)<+\infty$ for all $u>-\delta$ because $(\bar\varphi,\bar\eta)$ is feasible for such $u$. By the discussion in Remark \ref{rmk:J_lower_bound}, $v(u)>-\infty$. Hence $v$ is finite on an open interval containing $0$. As a finite
convex function on an open interval, $v$ is continuous at $0$ and
$\partial v(0)\neq\emptyset$.

Let $\xi\in\partial v(0)$. Since relaxing the constraint can only decrease the value, $v$ is nonincreasing, and therefore $\xi\le 0$. Define $\lambda^*:=-\xi\ge 0$. The subgradient inequality gives, for all $u$,
\[
v(u)\ge v(0)-\lambda^* u.
\]
For any $(\varphi,\eta)\in\mathcal A\times E$, taking $u=C(\varphi,\eta)$ yields
\[
J(\varphi)
\ge
v(C(\varphi,\eta))
\ge
p^*-\lambda^* C(\varphi,\eta),
\]
where the first inequality holds because $(\varphi,\eta)$ is feasible for the perturbed constraint level $u=C(\varphi,\eta)$. Therefore
\[
J(\varphi)+\lambda^* C(\varphi,\eta)\ge p^*
\]
for all $(\varphi,\eta)$. Taking the infimum over $(\varphi,\eta)\in\mathcal A\times E$ gives $q(\lambda^*)\ge p^*$. By Proposition \ref{prop:weak_duality}, $q(\lambda^*)\le p^*$. Hence $q(\lambda^*)=p^*$, so $d^*=p^*$ and $\lambda^*$ is a dual optimizer. 
\end{proof}

Similar to Proposition \ref{prop:eta_compactification}, we provide a bound for the multiplier $\lambda$, which is useful in the numerical algorithm. 
\begin{proposition}\label{prop:bound_lambda}
Under Assumptions \ref{ass:market},\ref{ass:action_set}, \ref{ass:regularity} and \ref{ass:slater},
let $\bar \varphi\in \mathcal A$ be chosen as in Assumption \ref{ass:slater}, and let $J_{\rm low}$ be a lower bound of $J$ on $\mathcal A$ from~Remark \ref{rmk:J_lower_bound}.  
Define
\[
\Lambda:=\frac{J(\bar\varphi)-J_{\rm low}}{\delta} \in [0, \infty).
\] Then every dual optimizer $\lambda^*$ satisfies
$0\le \lambda^*\le \Lambda$.
\end{proposition}

\begin{proof}
The lower bound $\lambda^*\ge 0$ follows from dual feasibility.
By Theorem \ref{thm:strong_duality},
$q(\lambda^*)=p^*$.
Since $J(\varphi)\ge J_{\rm low}$ for all $\varphi\in\mathcal A$, we have
$p^*\ge J_{\rm low}$. On the other hand, by definition of the dual
function, choosing $\bar \varphi, \bar \eta$ as in Remark \ref{rmk:eta_bar},  
\[
q(\lambda^*)
=
\inf_{\varphi\in\mathcal A,\eta\in E}
\mathcal L(\varphi,\eta,\lambda^*)
\le
\mathcal L(\bar\varphi,\bar\eta,\lambda^*).
\]
Therefore,
\[
q(\lambda^*)
\le
J(\bar\varphi)+\lambda^* C(\bar\varphi,\bar\eta)
\le
J(\bar\varphi)-\lambda^*\delta .
\]
Combining the inequalities gives
\[
J_{\rm low}\le p^*=q(\lambda^*)\le J(\bar\varphi)-\lambda^*\delta .
\]
Hence
\[
\lambda^*\le
\frac{J(\bar\varphi)-J_{\rm low}}{\delta}=\Lambda .
\]
This proves the claim.
\end{proof}

The following corollary explains how a primal optimizer can be recovered once a dual optimizer has been found. It also shows that, under strict convexity, the control component recovered from any Lagrangian minimizer is the unique primal optimal control.

\begin{corollary}[Recovery of a primal optimizer]\label{cor:saddle_optimality} Suppose Assumptions \ref{ass:market},\ref{ass:action_set}, \ref{ass:regularity} and \ref{ass:slater} hold. 
Let $\lambda^\star$ be a dual optimizer. Suppose
\[
(\varphi^\star,\eta^\star)
\in
\arg\min_{\varphi\in\mathcal A,\eta\in E}
\mathcal L(\varphi,\eta,\lambda^\star)
\]
and
\[
C(\varphi^\star,\eta^\star)\le0,
\qquad
\lambda^\star C(\varphi^\star,\eta^\star)=0.
\]
Then $(\varphi^\star,\eta^\star)$ is a primal optimizer.

If in addition Assumption \ref{ass:strict_convexity} also holds,
then any Lagrangian minimizer
\[
(\widehat\varphi,\widehat\eta)
\in
\arg\min_{\varphi\in\mathcal A,\eta\in E}
\mathcal L(\varphi,\eta,\lambda^\star)
\]
recovers the unique primal optimal control.
\end{corollary}

\begin{proof}
Since $(\varphi^\star,\eta^\star)$ minimizes the Lagrangian at
$\lambda^\star$,
$
\mathcal L(\varphi^\star,\eta^\star,\lambda^\star)
=
q(\lambda^\star).
$
By strong duality,
$
q(\lambda^\star)=p^\star.
$
Using complementary slackness,
\[
J(\varphi^\star)
=
J(\varphi^\star)+\lambda^\star C(\varphi^\star,\eta^\star)
=
\mathcal L(\varphi^\star,\eta^\star,\lambda^\star)
=
p^\star.
\]
Since $C(\varphi^\star,\eta^\star)\le0$, the pair is primal feasible and
attains the primal value. Hence it is a primal optimizer.

It remains to prove the recovery statement. Let
$(\varphi^{\rm p},\eta^{\rm p})\in\mathcal A\times E$ be a primal optimizer,
whose existence follows from Proposition~\ref{prop:primal_existence}. Since
$(\varphi^{\rm p},\eta^{\rm p})$ is feasible,
\[
\mathcal L(\varphi^{\rm p},\eta^{\rm p},\lambda^\star)
=
J(\varphi^{\rm p})+\lambda^\star C(\varphi^{\rm p},\eta^{\rm p})
\le
J(\varphi^{\rm p})
=
p^\star.
\]
On the other hand, since $\lambda^\star$ is dual optimal, strong duality gives
$q(\lambda^\star)=p^\star$, and by definition of $q$,
\[
q(\lambda^\star)
\le
\mathcal L(\varphi^{\rm p},\eta^{\rm p},\lambda^\star).
\]
Therefore
$\mathcal L(\varphi^{\rm p},\eta^{\rm p},\lambda^\star)
=
q(\lambda^\star)$,
so $(\varphi^{\rm p},\eta^{\rm p})$ is a Lagrangian minimizer at
$\lambda^\star$.

By Proposition~\ref{prop:lagrangian_minimizer_existence}, under the strict
convexity of $J$, the control component of the
Lagrangian minimizer at $\lambda^\star$ is unique. Hence any Lagrangian
minimizer $(\widehat\varphi,\widehat\eta)$ at $\lambda^\star$ satisfies
$\widehat\varphi=\varphi^{\rm p}$.
Thus $\widehat\varphi$ is the unique primal optimal control.
\end{proof}

\section{Algorithm and Analysis}\label{sec:algorithm}
\subsection{Nested Bisection-Golden Search Method}\label{ssec:lambda and eta}
In this section, motivated by the dual formulation in Section \ref{ssec:dual}, we propose a nested bisection method for the portfolio optimization problem with CVaR constraint, by solving the compactified dual problem. By Propositions~\ref{prop:eta_compactification} and~\ref{prop:bound_lambda}, it is enough to solve the dual problem
\[ d^\star = \sup_{\lambda\in[0,\Lambda]}q(\lambda), \qquad q(\lambda) = \inf_{\varphi\in\mathcal A,\eta\in E} \mathcal L(\varphi,\eta,\lambda). \] 
The proposed algorithm consists of a control oracle for fixed
$(\lambda,\eta)$, an inner golden-section search over $\eta$, and an
outer bisection over $\lambda$.

Throughout this section, Assumptions~\ref{ass:market}--\ref{ass:regularity},
\ref{ass:slater}, and~\ref{ass:strict_convexity} are in force.
Assumption~\ref{ass:strict_convexity} ensures that both $J$ and $\mathcal L(\cdot,\eta,\lambda)$ are strictly convex in the control component. Consequently, every fixed-$(\lambda,\eta)$ control problem and every joint inner problem $ q(\lambda)$ have a unique control component.

\paragraph{Control oracle.}

For fixed $(\lambda,\eta)$, define the modified terminal cost
\[
\Psi_{\lambda,\eta}(w)
:=
g(w)+\frac{\lambda}{1-\alpha}(\ell(w)-\eta)^+.
\]
Then the fixed-$(\lambda,\eta)$ subproblem is
\begin{equation}\label{eq:fixed_lambda_eta}
  \inf_{\varphi\in\mathcal A} \mathcal L(\varphi,\eta,\lambda)=
  \inf_{\varphi\in\mathcal A}
\mathbb E\left[
\int_0^T f(t,W_t^\varphi,\varphi_t)\,dt
+
\Psi_{\lambda,\eta}(W_T^\varphi)
\right]
+\lambda(\eta-c).
\end{equation}
The last term $\lambda(\eta-c)$ is independent of $\varphi$, so the control
oracle only needs to solve a standard stochastic control problem with terminal
cost $\Psi_{\lambda,\eta}$.
We denote its output by
\[
(\varphi_{\lambda,\eta}, V_\lambda(\eta))
:=
\mathsf{Control}(\lambda,\eta),
\]
where $\varphi_{\lambda,\eta}=\arg\min_{\varphi\in\mathcal A}
\mathcal L(\varphi,\eta,\lambda)$ is the unique solution of \eqref{eq:fixed_lambda_eta} under Assumption \ref{ass:strict_convexity}, and $V_\lambda(\eta)
=
\mathcal L(\varphi_{\lambda,\eta},\eta,\lambda)=\inf_{\varphi\in\mathcal A} \mathcal L(\varphi,\eta,\lambda)$ is the corresponding optimal value.

\paragraph{Inner golden-section search over the CVaR threshold.}
We next construct a method for evaluating the dual function
$q$ defined in \eqref{eq:dual_function}. For any fixed $\lambda>0$,
\[
\begin{aligned}
q(\lambda)
&=
\inf_{\varphi\in\mathcal A,\ \eta\in E}
\mathcal L(\varphi,\eta,\lambda) \\
&=
\inf_{\eta\in E}
\inf_{\varphi\in\mathcal A}
\mathcal L(\varphi,\eta,\lambda)
=
\inf_{\eta\in E}V_\lambda(\eta).
\end{aligned}
\]
Thus, evaluating $q(\lambda)$ reduces to a one-dimensional minimization
of $V_\lambda$ over the compact interval $E$.

\begin{proposition}[Properties of the inner value function]
\label{prop:inner_value_properties}
Suppose that Assumptions~\ref{ass:market},
\ref{ass:action_set}, \ref{ass:regularity}, and~\ref{ass:slater} hold. Then, for every
$\lambda\in[0,\Lambda]$, the function $V_\lambda$ is finite and convex
on $E$ and moreover,
\[
|V_\lambda(\eta)-V_\lambda(\eta')|
\le
\lambda\kappa_\alpha|\eta-\eta'|,
\qquad
\eta,\eta'\in E,
\]
where
$\kappa_\alpha
:=
\max\left\{1,\frac{\alpha}{1-\alpha}\right\}$.
Consequently,
$\mathcal E_\lambda
:=
\arg\min_{\eta\in E}V_\lambda(\eta)$
is a nonempty closed interval.
\end{proposition}
\begin{proof}
Let $\bar\varphi$ be the finite-cost strategy in
Assumption~\ref{ass:slater}. Since $C(\bar\varphi,\eta)$ is finite,
$V_\lambda(\eta)\le\mathcal L(\bar\varphi,\eta,\lambda)$.
Moreover, by Remark~\ref{rmk:J_lower_bound} and the nonnegativity of
the hinge term,
\[
\mathcal L(\varphi,\eta,\lambda)
\ge
J_{\rm low}+\lambda(\eta-c),
\qquad \varphi\in\mathcal A.
\]
Hence $V_\lambda(\eta)>-\infty$.
For fixed $\lambda$, the affine dependence of $W^\varphi$ on $\varphi$,
the convexity of $f$, $g$, and $\ell$, and the convexity of $\mathcal A$
imply that $(\varphi,\eta)
\mapsto
\mathcal L(\varphi,\eta,\lambda)$
is jointly convex. Partial minimization over $\varphi$ therefore shows
that $V_\lambda$ is convex.

For fixed $x\in\mathbb R$, define
$r_x(\eta)
:=
\eta+\frac{1}{1-\alpha}(x-\eta)^+$.
Every subgradient of $r_x$ has an absolute value bounded by
$\kappa_\alpha
:=
\max\left\{1,\frac{\alpha}{1-\alpha}\right\},
$ and hence
$|r_x(\eta)-r_x(\eta')|
\le
\kappa_\alpha|\eta-\eta'|$.
It follows that, uniformly over $\varphi\in\mathcal A$,
\[
\left|
\mathcal L(\varphi,\eta,\lambda)
-
\mathcal L(\varphi,\eta',\lambda)
\right|
\le
\lambda\kappa_\alpha|\eta-\eta'|.
\]
Taking the infimum over $\varphi$ in both directions gives
\[
|V_\lambda(\eta)-V_\lambda(\eta')|
\le
\lambda\kappa_\alpha|\eta-\eta'|.
\]
Thus $V_\lambda$ is Lipschitz hence continuous. Since $E$ is compact,
$\mathcal E_\lambda$ is nonempty and closed, and its convexity follows
from the convexity of $V_\lambda$.
\end{proof}

Since $V_\lambda$ is convex, it decreases before
to the left of 
its minimizer set and
increases to the right of it. 
This allows us to locate a minimizer using function
values only. Starting from an interval $[\eta_{\rm low},\eta_{\rm high}]\subseteq E$ containing
$\mathcal E_\lambda$, define
\[
\rho:=\frac{\sqrt5-1}{2},
\qquad
\eta_1=\eta_{\rm high}-\rho(\eta_{\rm high}-\eta_{\rm low}),
\qquad
\eta_2=\eta_{\rm low}+\rho(\eta_{\rm high}-\eta_{\rm low}).
\]
We evaluate $V_\lambda$ at these two interior points $\eta_1,\eta_2$ using the control oracle
\[
(\varphi_{\lambda,\eta},V_\lambda(\eta))
=
\mathsf{Control}(\lambda,\eta).
\]
If $V_\lambda(\eta_1)\le V_\lambda(\eta_2)$, convexity implies that
the interval $[\eta_{\rm low},\eta_2]$ contains a minimizer, so the
right-hand portion may be discarded. 
Otherwise, the interval
$[\eta_1,\eta_{\rm high}]$ contains a minimizer, and the
left-hand portion may be discarded. The points are chosen according
to the golden ratio so that one previous function evaluation can be reused
after each interval reduction. Thus, after the initial two evaluations, each
iteration requires only one new call to the control oracle. The algorithm is summarized as follows.

\begin{algorithm}[H]
\small
\SetAlgoLined
\KwData{Multiplier $\lambda>0$, interval
$E=[\eta_{\rm low},\eta_{\rm high}]$, number of reductions $N_\eta\ge 1$}
\KwResult{Approximate inner solution
$(\widehat\varphi_\lambda,\widehat\eta_\lambda)$}

$\rho\gets(\sqrt5-1)/2$\;
$\eta_1\gets
\eta_{\rm high}
-\rho(\eta_{\rm high}-\eta_{\rm low})$\;
$\eta_2\gets
\eta_{\rm low}
+\rho(\eta_{\rm high}-\eta_{\rm low})$\;
$(\varphi_i,v_i)\gets\mathsf{Control}(\lambda,\eta_i)$,
$i=1,2$\;

\For{$j=1,\ldots,N_\eta$}{
\eIf{$v_1\le v_2$}{
    $\eta_{\rm high}\gets\eta_2$\;
    $(\eta_2,\varphi_2,v_2)\gets(\eta_1,\varphi_1,v_1)$\;
    $\eta_1\gets
    \eta_{\rm high}
    -\rho(\eta_{\rm high}-\eta_{\rm low})$\;
    $(\varphi_1,v_1)\gets\mathsf{Control}(\lambda,\eta_1)$\;
}{
    $\eta_{\rm low}\gets\eta_1$\;
    $(\eta_1,\varphi_1,v_1)\gets(\eta_2,\varphi_2,v_2)$\;
    $\eta_2\gets
    \eta_{\rm low}
    +\rho(\eta_{\rm high}-\eta_{\rm low})$\;
    $(\varphi_2,v_2)\gets\mathsf{Control}(\lambda,\eta_2)$\;
}
}
\Return{$(\varphi_1,\eta_1)$ if $v_1\le v_2$;
otherwise $(\varphi_2,\eta_2)$}\;
\caption{$\mathsf{EtaGoldenSearch}(\lambda)$}
\label{alg:eta_golden_search}
\end{algorithm}

\paragraph{Outer bisection over the Lagrange multiplier.}
The inner golden-section search via Algorithm~\ref{alg:eta_golden_search} approximately evaluates $q(\lambda)$ for each fixed multiplier $\lambda$. We now maximize the concave dual function $q$ over $[0,\Lambda]$. The key observation is that its derivative is given by the constraint residual, which we define in the following two cases of $\lambda$:

For every $\lambda>0$, let 
\[ (\varphi_\lambda,\eta_\lambda) \in \arg\min_{\varphi\in\mathcal A,\eta\in E} \mathcal L(\varphi,\eta,\lambda), \] 
and define the residual 
\[ R(\lambda):=C(\varphi_\lambda,\eta_\lambda). \] 
Under Assumption~\ref{ass:strict_convexity}, the control component
$\varphi_\lambda$ is unique. Although the optimizing threshold
$\eta_\lambda$ may not be unique, the value of the residual is unique.
Indeed, if $\eta_\lambda$ and $\widetilde\eta_\lambda$ are both optimal
thresholds for the same control, then
\[
J(\varphi_\lambda)
+\lambda C(\varphi_\lambda,\eta_\lambda)
=
J(\varphi_\lambda)
+\lambda C(\varphi_\lambda,\widetilde\eta_\lambda).
\]
Since $\lambda>0$,
$C(\varphi_\lambda,\eta_\lambda)
=
C(\varphi_\lambda,\widetilde\eta_\lambda)$.
Thus, $R(\lambda)$ is well defined.

For $\lambda=0$, we define $(\varphi_0, \eta_0)$ as follows. Solving the unconstrained problem yields $\varphi_0$. Next, we choose
\[
\eta_0\in \arg\min_{\eta\in E} \left\{ \eta+\frac{1}{1-\alpha} \mathbb E[(\ell(W_T^{\varphi_0})-\eta)^+] \right\}, 
\]
which yields the residual
\[ R(0)=C(\varphi_0,\eta_0) = \operatorname{CVaR}_\alpha(\ell(W_T^{\varphi_0}))-c.
\]
If $R(0)\le0$, then the unconstrained optimizer is feasible and the optimal multiplier is $\lambda^\star=0$. Otherwise, the constraint is active, and we search for a root of $R$ on $(0,\Lambda]$.

\begin{proposition}[Dual residual]
\label{prop:dual_residual}
Suppose that Assumptions~\ref{ass:market},
\ref{ass:action_set}, \ref{ass:regularity}, 
\ref{ass:strict_convexity}, and~\ref{ass:slater} hold. Then $R$ is nonincreasing and continuous on $[0,\infty)$, and
therefore uniformly continuous on $[0,\Lambda]$.
Moreover,  $q$ is differentiable on
$(0,\infty)$, with
$q'(\lambda)=R(\lambda)$ for all $\lambda>0$.
\end{proposition}
\begin{proof}
By Proposition~\ref{prop:eta_compactification}, there exists
$M_C<\infty$ such that
\[
|C(\varphi,\eta)|\le M_C,
\qquad
(\varphi,\eta)\in\mathcal A\times E.
\]
Together with the finite-cost strategy provided by
Assumption~\ref{ass:slater} and the lower bound on $J$, this implies
that $q$ is finite. Moreover,
\[
|q(\lambda_2)-q(\lambda_1)|
\le
M_C|\lambda_2-\lambda_1|,
\qquad
\lambda_1,\lambda_2\ge0,
\]
so $q$ is continuous.

Let $0\le\lambda_1<\lambda_2$, and choose a Lagrangian minimizer at
each multiplier, using $(\varphi_0,\eta_0)$ when $\lambda_1=0$.
Optimality gives
\[
q(\lambda_2)
\le
q(\lambda_1)
+
(\lambda_2-\lambda_1)R(\lambda_1)
\]
and
\[
q(\lambda_1)
\le
q(\lambda_2)
-
(\lambda_2-\lambda_1)R(\lambda_2).
\]
Therefore,
\begin{equation}
\label{eq:dual_secant_bound}
R(\lambda_2)
\le
\frac{q(\lambda_2)-q(\lambda_1)}
{\lambda_2-\lambda_1}
\le
R(\lambda_1).
\end{equation}
In particular, $R$ is nonincreasing.

Next, we show continuity. Let $\lambda_n\to\lambda$ monotonically and
choose
\[
(\varphi_n,\eta_n)
\in
\argmin_{\varphi\in\mathcal A,\ \eta\in E}
\mathcal L(\varphi,\eta,\lambda_n).
\]
By weak compactness, along a subsequence,
$\varphi_n\rightharpoonup\bar\varphi$ and
$\eta_n\to\bar\eta$.
The weak lower semicontinuity of the Lagrangian, the boundedness of
$C$, and the continuity of $q$ imply that
\[
(\bar\varphi,\bar\eta)
\in
\argmin_{\varphi\in\mathcal A,\ \eta\in E}
\mathcal L(\varphi,\eta,\lambda).
\]

Suppose first that $\lambda_n\downarrow\lambda$ and write
$R(\lambda_n)\to r$. Monotonicity gives $r\le R(\lambda)$. By the weak
lower semicontinuity of $C$,
$C(\bar\varphi,\bar\eta)\le r$.
If $\lambda>0$, then
$C(\bar\varphi,\bar\eta)=R(\lambda)$. If $\lambda=0$, the definition
of $\eta_0$ gives
$R(0)\le C(\bar\varphi,\bar\eta)$.
In either case, $R(\lambda)\le r$, and hence
$r=R(\lambda)$. Thus $R$ is right-continuous, including at zero.

Now suppose that $\lambda_n\uparrow\lambda$ with $\lambda>0$, and
write $R(\lambda_n)\to r$. Monotonicity gives $r\ge R(\lambda)$.
Moreover,
\[
J(\varphi_n)
=
q(\lambda_n)-\lambda_nR(\lambda_n).
\]
Using the weak lower semicontinuity of $J$ gives
\[
q(\lambda)-\lambda R(\lambda)
=
J(\bar\varphi)
\le
\liminf_{n\to\infty}J(\varphi_n)
=
q(\lambda)-\lambda r.
\]
Since $\lambda>0$, this implies $r\le R(\lambda)$. Therefore
$r=R(\lambda)$, proving left continuity.

Finally, for $\lambda>0$, applying
\eqref{eq:dual_secant_bound} to $\lambda$ and $\lambda+h$ yields
\[
R(\lambda+h)
\le
\frac{q(\lambda+h)-q(\lambda)}{h}
\le
R(\lambda).
\]
Letting $h\downarrow0$ and using right continuity gives
$q'_+(\lambda)=R(\lambda)$.
Similarly,
\[
R(\lambda)
\le
\frac{q(\lambda)-q(\lambda-h)}{h}
\le
R(\lambda-h),
\]
and left continuity gives
$q'_-(\lambda)=R(\lambda)$.
Thus $q$ is differentiable and
$q'(\lambda)=R(\lambda)$ for $\lambda>0$.
\end{proof}

In the active-constraint case, $\lambda^\star\ne 0$. Define the dual optimal set
$\mathcal Z
:=
\{\lambda\in[0,\Lambda]:R(\lambda)=0\}$. By Theorem~\ref{thm:strong_duality},
Proposition~\ref{prop:bound_lambda}, and the differentiability of
$q$, the set $Z$ is nonempty and coincides with the set of dual
optimizers.
Since $R$ is nonincreasing,
$R(\lambda)>0$
means that the multiplier is too small, so the search must move to the
right. Similarly,
$R(\lambda)<0$
means that the multiplier is too large, so the search must move to the
left. Thus, the root of $R$ can be located by bisection.

For a queried multiplier $\lambda$, Algorithm \ref{alg:eta_golden_search} returns
$(\widehat\varphi_\lambda,\widehat\eta_\lambda)
=
\mathsf{EtaGoldenSearch}(\lambda)$.
We use the corresponding residual
$\widehat R(\lambda)
:=
C(\widehat\varphi_\lambda,\widehat\eta_\lambda)$
in the outer bisection. Starting from $[0,\Lambda]$, the method evaluates
$\widehat R$ at the midpoint of the current interval. If
$\widehat R(\lambda)>0$, it retains the right half; if
$\widehat R(\lambda)<0$, it retains the left half. The main algorithm is shown below.

\begin{algorithm}[H]
\small
\SetAlgoLined
\KwData{Multiplier bound $\Lambda$, maximum iteration $N_\lambda\ge 1$}
\KwResult{Approximate primal-dual solution $(\widehat\varphi,\widehat\eta,\widehat\lambda)$}

Solve the unconstrained problem and compute $R(0)$\;
\If{$R(0)\le0$}{
   \Return{$(\varphi_0,\eta_0,0)$}\;
}

$\lambda_{\rm low}\gets 0$, $\lambda_{\rm high}\gets\Lambda$\;

\For{$k=0,1,\ldots,N_\lambda-1$}{
   $\lambda_{\rm mid}\gets(\lambda_{\rm low}+\lambda_{\rm high})/2$\;
   $(\varphi_{\rm mid},\eta_{\rm mid})\gets\mathsf{EtaGoldenSearch}(\lambda_{\rm mid})$\;
   $R_{\rm mid}\gets C(\varphi_{\rm mid},\eta_{\rm mid})$\;

   \If{$R_{\rm mid}=0$}{
       \Return{$(\varphi_{\rm mid},\eta_{\rm mid},\lambda_{\rm mid})$}\;
   }

   \eIf{$R_{\rm mid}>0$}{
       $\lambda_{\rm low}\gets\lambda_{\rm mid}$\;
   }{
       $\lambda_{\rm high}\gets\lambda_{\rm mid}$\;
   }
}
\Return{$(\varphi_{\rm mid},\eta_{\rm mid},\lambda_{\rm mid})$}\;
\caption{Nested bisection--golden search algorithm}
\label{alg:lambda_bisection}
\end{algorithm}

\subsection{Convergence Analysis}\label{ssec:algorithm_convergence}

We analyze the performance of the main algorithm under exact evaluations of the control oracle. We first study the convergence of the inner golden-search over the CVaR threshold (Algorithm \ref{alg:eta_golden_search}) and then establish the convergence of Algorithm \ref{alg:lambda_bisection}.

For a nonempty set $S\subset\mathbb R$, we write
\[
\operatorname{dist}(\eta,S)
:=
\inf_{\eta'\in S}|\eta-\eta'|.
\]

\begin{proposition}[Convergence of the inner golden-section search]
\label{prop:inner_search_convergence}
Let $\rho:=\frac{\sqrt5-1}{2}$.
Suppose that Assumptions~\ref{ass:market},
\ref{ass:action_set}, \ref{ass:regularity}, and~\ref{ass:slater} hold, and that the
control and value used by Algorithm \ref{alg:eta_golden_search} are exact. After $N_\eta$ interval
reductions, Algorithm~\ref{alg:eta_golden_search} returns
$(\widehat\varphi_\lambda,\widehat\eta_\lambda)$ satisfying
\[
\operatorname{dist}
\bigl(\widehat\eta_\lambda,\mathcal E_\lambda\bigr)
\le
|E|\rho^{N_\eta},
\]
and
\[
0
\le
V_\lambda(\widehat\eta_\lambda)-q(\lambda)
\le
\lambda\kappa_\alpha |E|\rho^{N_\eta}.
\]
\end{proposition}

\begin{proof}
By Proposition~\ref{prop:inner_value_properties},
$V_\lambda$ is convex and
$\mathcal E_\lambda$ is a nonempty closed interval.

Consider a current search interval
$[\eta_{\rm low},\eta_{\rm high}]$ that intersects
$\mathcal E_\lambda$, and let $\eta_1<\eta_2$ be the two golden-section
points. If $V_\lambda(\eta_1)\le V_\lambda(\eta_2)$,
convexity implies that a minimizer exists in
$[\eta_{\rm low},\eta_2]$, so the algorithm may discard
$(\eta_2,\eta_{\rm high}]$. Similarly, if
$V_\lambda(\eta_1)>V_\lambda(\eta_2)$,
a minimizer exists in
$[\eta_1,\eta_{\rm high}]$. Thus every update retains an interval
intersecting $\mathcal E_\lambda$.

Each golden-section update reduces the interval length by the factor
$\rho$. Therefore, after $N_\eta$ reductions, the remaining interval has
length $|E|\rho^{N_\eta}$
and still intersects $\mathcal E_\lambda$. Since
$\widehat\eta_\lambda$ lies in this interval,
\[
\operatorname{dist}
\bigl(\widehat\eta_\lambda,\mathcal E_\lambda\bigr)
\le
|E|\rho^{N_\eta}.
\]

Choose
$\eta_\lambda^\star\in\mathcal E_\lambda$ such that
$|\widehat\eta_\lambda-\eta_\lambda^\star|
=
\operatorname{dist}
\bigl(\widehat\eta_\lambda,\mathcal E_\lambda\bigr)$.
Since
$q(\lambda)
=
\min_{\eta\in E}V_\lambda(\eta)
=
V_\lambda(\eta_\lambda^\star)$,
the Lipschitz estimate in
Proposition~\ref{prop:inner_value_properties} yields
\[
\begin{aligned}
0
&\le
V_\lambda(\widehat\eta_\lambda)-q(\lambda) \le
\lambda\kappa_\alpha
|\widehat\eta_\lambda-\eta_\lambda^\star| \le
\lambda\kappa_\alpha |E|\rho^{N_\eta}.
\end{aligned}
\]
\end{proof}

\begin{theorem}
\label{thm:nested_search_convergence}
Suppose that Assumptions~\ref{ass:market},
\ref{ass:action_set}, \ref{ass:regularity},
\ref{ass:slater}, and~\ref{ass:strict_convexity} hold, and that the
control, value, and residual evaluations used by the algorithm are exact.
Consider the active-constraint case $R(0)>0$, and let
$(\widehat\varphi,\widehat\eta,\widehat\lambda)$ be the output of the
Algorithm \ref{alg:lambda_bisection} with $N_\lambda$ outer
iterations and $N_\eta$ inner iterations.

Let $\rho:=\frac{\sqrt5-1}{2}$.
If $N_\lambda\to\infty$ and 
$2^{N_\lambda}\rho^{N_\eta}\to0$,
then
\[
\operatorname{dist}(\widehat\lambda,\mathcal Z)\to0,
\qquad
C(\widehat\varphi,\widehat\eta)\to0,
\qquad
J(\widehat\varphi)\to p^\star.
\]
Moreover,
\[
\widehat\varphi
\rightharpoonup
\varphi^\star
\qquad\text{weakly in }L^2_{\mathbb F},
\]
where $\varphi^\star$ is the unique primal optimal control.
\end{theorem}

\begin{proof}
For every multiplier
$\lambda$ queried by the outer bisection, let $ (\widehat\varphi_\lambda,\widehat\eta_\lambda) $ denote the output of the inner bisection and the associated exact control oracle. Then by Proposition~\ref{prop:inner_search_convergence}, 
\[
0
\le
\mathcal L(
\widehat\varphi_\lambda,
\widehat\eta_\lambda,
\lambda)
-q(\lambda)
\le
\lambda\kappa_\alpha |E|\rho^{N_\eta}
\le
\Lambda\kappa_\alpha |E|\rho^{N_\eta}.
\]

For any $\mu\in[0,\Lambda]$, the definition of $q$ and the affine
dependence of the Lagrangian on the multiplier give
\[
\begin{aligned}
q(\mu)
&\le
\mathcal L(
\widehat\varphi_\lambda,
\widehat\eta_\lambda,
\mu) \\
&=
\mathcal L(
\widehat\varphi_\lambda,
\widehat\eta_\lambda,
\lambda)
+
(\mu-\lambda)
C(\widehat\varphi_\lambda,\widehat\eta_\lambda) \\
&\le
q(\lambda)
+
\Lambda\kappa_\alpha |E|\rho^{N_\eta}
+
(\mu-\lambda)
C(\widehat\varphi_\lambda,\widehat\eta_\lambda).
\end{aligned}
\]
Among all multipliers queried during the $N_\lambda$ outer iterations and
at the final midpoint, the smallest possible distance from the endpoints of
$[0,\Lambda]$ is $\Lambda2^{-(N_\lambda+1)}$. Hence every queried
multiplier $\lambda$ satisfies
\[
\lambda\pm\Lambda2^{-(N_\lambda+1)}\in[0,\Lambda].
\]
Taking
$\mu=\lambda+\Lambda2^{-(N_\lambda+1)}$
in the preceding inequality gives
\[
C(\widehat\varphi_\lambda,\widehat\eta_\lambda)
\ge
\frac{
q\!\left(\lambda+\Lambda2^{-(N_\lambda+1)}\right)-q(\lambda)
}{
\Lambda2^{-(N_\lambda+1)}
}
-
2\kappa_\alpha |E|2^{N_\lambda}\rho^{N_\eta}.
\]
Since $q$ is concave and $q'=R$,
\[
\frac{
q\!\left(\lambda+\Lambda2^{-(N_\lambda+1)}\right)-q(\lambda)
}{
\Lambda2^{-(N_\lambda+1)}
}
\ge
R\!\left(\lambda+\Lambda2^{-(N_\lambda+1)}\right).
\]
Similarly, taking $\mu=\lambda-\Lambda2^{-(N_\lambda+1)}$
and using concavity yields
\[
C(\widehat\varphi_\lambda,\widehat\eta_\lambda)
\le
R\!\left(\lambda-\Lambda2^{-(N_\lambda+1)}\right)
+
2\kappa_\alpha |E|2^{N_\lambda}\rho^{N_\eta}.
\]
Combining the two inequalities proves
\[
\begin{aligned}
R\!\left(\lambda+\Lambda2^{-(N_\lambda+1)}\right)
-
2\kappa_\alpha |E|2^{N_\lambda}\rho^{N_\eta}
&\le
C(\widehat\varphi_\lambda,\widehat\eta_\lambda) \\
&\le
R\!\left(\lambda-\Lambda2^{-(N_\lambda+1)}\right)
+
2\kappa_\alpha |E|2^{N_\lambda}\rho^{N_\eta}.
\end{aligned}
\]
Consequently,
\[
\begin{aligned}
\left|
C(\widehat\varphi_\lambda,\widehat\eta_\lambda)
-
R(\lambda)
\right|
\le{}&
\sup_{\substack{\lambda_1,\lambda_2\in[0,\Lambda]\\
|\lambda_1-\lambda_2|
\le\Lambda2^{-(N_\lambda+1)}}}
|R(\lambda_1)-R(\lambda_2)| +
2\kappa_\alpha |E|2^{N_\lambda}\rho^{N_\eta}.
\end{aligned}
\]
By Proposition~\ref{prop:dual_residual}, with
$R(0):=q'_+(0)
=
\min_{\eta\in E}C(\varphi_0,\eta)$
and $R(\Lambda):=q'_-(\Lambda)$, the residual $R$ is continuous on
$[0,\Lambda]$. Hence it is uniformly continuous.
Therefore, if
$N_\lambda\to\infty$ and 
$2^{N_\lambda}\rho^{N_\eta}\to 0$,
then
\begin{equation}\label{eq:uniform_residual_approximation}
   \sup_{\lambda}
\left|
C(\widehat\varphi_\lambda,\widehat\eta_\lambda)
-
R(\lambda)
\right|
\to0, 
\end{equation}
where the supremum is over the multipliers queried by the outer bisection.

We next prove that
$\operatorname{dist}(\widehat\lambda,\mathcal Z)\to0$.
Fix $\varepsilon>0$. Since $R$ is continuous and
$\mathcal Z=\{\lambda\in[0,\Lambda]:R(\lambda)=0\}$,
compactness implies that
\[
\inf_{\substack{\lambda\in[0,\Lambda]\\
\operatorname{dist}(\lambda,\mathcal Z)\ge\varepsilon}}
|R(\lambda)|>0,
\]
provided that the set over which the infimum is taken is nonempty. By \eqref{eq:uniform_residual_approximation}, for sufficiently large
$N_\lambda$ and $N_\eta$,
$C(\widehat\varphi_\lambda,\widehat\eta_\lambda)$ and $
R(\lambda)$
have the same sign at every queried midpoint satisfying
$\operatorname{dist}(\lambda,\mathcal Z)\ge\varepsilon$.

We claim that, throughout the outer bisection, the current bracket intersects
the closed $\varepsilon$-neighborhood of $\mathcal Z$. This is true for the
initial bracket $[0,\Lambda]$. Suppose it holds for the current bracket, and
let $\lambda$ be its midpoint. If $\operatorname{dist}(\lambda,\mathcal Z)\le \varepsilon$,
then either half selected by the algorithm contains $\lambda$ as an
endpoint, and hence still intersects the $\varepsilon$-neighborhood of
$\mathcal Z$. 

Otherwise,
$\operatorname{dist}(\lambda,\mathcal Z)>\varepsilon$,
so the approximate residual has the same sign as $R(\lambda)$. Since $R$
is nonincreasing, its zero set $\mathcal Z$ is an interval. If $\lambda$
lies to the left of the $\varepsilon$-neighborhood of $\mathcal Z$, then
$R(\lambda)>0$, and the algorithm retains the right half of the bracket,
which still intersects that neighborhood. Similarly, if $\lambda$ lies to
the right of the $\varepsilon$-neighborhood, then $R(\lambda)<0$, and the
algorithm retains the left half, which again intersects the neighborhood.
The claim therefore follows by induction.

Consequently, the final bracket contains some point $\widetilde\lambda$
such that
$\operatorname{dist}(\widetilde\lambda,\mathcal Z)\le\varepsilon$.
Since $\widehat\lambda$ also lies in the final bracket, whose length is at
most $\Lambda2^{-N_\lambda}$,
\[
\begin{aligned}
\operatorname{dist}(\widehat\lambda,\mathcal Z)
&\le
|\widehat\lambda-\widetilde\lambda|
+
\operatorname{dist}(\widetilde\lambda,\mathcal Z) \le
\Lambda2^{-N_\lambda}+\varepsilon.
\end{aligned}
\]
Taking the limit superior and then letting $\varepsilon\downarrow0$ gives
$\operatorname{dist}(\widehat\lambda,\mathcal Z)\to0$.

Because $R$ is continuous and vanishes on $\mathcal Z$,
$R(\widehat\lambda)\to0$. Applying \eqref{eq:uniform_residual_approximation} at the final multiplier yields
\[
\begin{aligned}
\left|
C(\widehat\varphi,\widehat\eta)
\right|
&\le
\left|
C(\widehat\varphi,\widehat\eta)
-
R(\widehat\lambda)
\right|
+
|R(\widehat\lambda)| \to0.
\end{aligned}
\]

Applying Proposition \ref{prop:inner_search_convergence} at $\widehat\lambda$ gives
\[
\mathcal L(
\widehat\varphi,
\widehat\eta,
\widehat\lambda)
\le
q(\widehat\lambda)
+
\Lambda\kappa_\alpha|E|\rho^{N_\eta}
\le
p^\star
+
\Lambda\kappa_\alpha|E|\rho^{N_\eta}.
\]
Therefore,
\[
J(\widehat\varphi)-p^\star
\le
\Lambda\kappa_\alpha|E|\rho^{N_\eta}
+
\widehat\lambda
\left|
C(\widehat\varphi,\widehat\eta)
\right|.
\]
On the other hand, for any $\lambda^\star\in\mathcal Z$, strong duality
implies
\[
p^\star
=
q(\lambda^\star)
\le
J(\widehat\varphi)
+
\lambda^\star
C(\widehat\varphi,\widehat\eta).
\]
Since
$\widehat\lambda,\lambda^\star\in[0,\Lambda]$, we conclude that
\[
\left|
J(\widehat\varphi)-p^\star
\right|
\le
\Lambda\kappa_\alpha|E|\rho^{N_\eta}
+
\Lambda
\left|
C(\widehat\varphi,\widehat\eta)
\right|
\to0.
\]

It remains to prove the convergence of the controls. Since
$\mathcal A$ is weakly compact in
$L^2_{\mathbb F}([0,T]\times\Omega;\mathbb R^m)$ and $E$ is compact,
every subsequence of
$(\widehat\varphi,\widehat\eta)$ admits a further subsequence, still
denoted by $(\widehat\varphi,\widehat\eta)$, such that
\[
\widehat\varphi\rightharpoonup\overline\varphi
\quad\text{weakly in }L^2_{\mathbb F},
\qquad
\widehat\eta\to\overline\eta\in E.
\]
By the weak lower semicontinuity of $J$ and $C$,
$J(\overline\varphi)
\le
\liminf J(\widehat\varphi)
=
p^\star$,
and
$C(\overline\varphi,\overline\eta)
\le
\liminf C(\widehat\varphi,\widehat\eta)
=
0$.
Thus $(\overline\varphi,\overline\eta)$ is primal feasible and attains the
primal value. Hence $\overline\varphi$ is a primal optimal control.
By Assumption~\ref{ass:strict_convexity}, the primal optimal control is
unique, and therefore $\overline\varphi=\varphi^\star$.
Thus every weakly convergent subsequence of $\widehat\varphi$ has limit
$\varphi^\star$. Since $\mathcal A$ is weakly compact, it follows that the
entire sequence satisfies
\[
\widehat\varphi
\rightharpoonup
\varphi^\star
\qquad
\text{weakly in }L^2_{\mathbb F}.
\]
\end{proof}

The strict-convexity condition in
Theorem~\ref{thm:nested_search_convergence} guarantees the weak convergence. Strong convergence follows under
the following stronger condition.

\begin{assumption}\label{ass:strong_convexity}
$J$ is
strongly convex on its effective domain with respect to the
$L^2_{\mathbb F}$ norm. That is, there exists a constant $\kappa>0$ such that for any
$\varphi,\psi\in\mathcal A$ satisfying
$J(\varphi)<\infty$,
$J(\psi)<\infty$,
and $\theta\in(0,1)$,
\[
\begin{aligned}
J\bigl(\theta\varphi+(1-\theta)\psi\bigr)
\le
\theta J(\varphi)+(1-\theta)J(\psi)
-
\frac{\kappa}{2}\theta(1-\theta)
\|\varphi-\psi\|_{L^2_{\mathbb F}}^2.
\end{aligned}
\]
\end{assumption}

Similar to Proposition \ref{prop:strict_convexity}, we also provide some explicit conditions to guarantee the strongly convex property of $J$.

\begin{proposition}\label{prop:strong_convexity}
    Suppose that Assumptions~\ref{ass:market},
\ref{ass:action_set}, and \ref{ass:regularity} hold.
Then the objective functional $J$ is strongly convex on its effective domain, if one of the following conditions hold:

\begin{enumerate}
    \item \label{ass:strong_convex_f} The function
$(w,a)\mapsto f(t,w,a)$ is strongly convex. That is,  
there exist a constant $\kappa_f>0$ 
such that, for almost
every $t\in[0,T]$, every $(w,a),(w',a')\in\mathbb R\times A$, and every
$\theta\in(0,1)$,
\[
\begin{aligned}
f\bigl(
t,\theta w+(1-\theta)w',
\theta a+(1-\theta)a'
\bigr)
&\le
\theta f(t,w,a)
+(1-\theta)f(t,w',a')
\\
&\qquad
-\frac{\kappa_f}{2}\theta(1-\theta)
\left(
|w-w'|^2+|a-a'|^2
\right).
\end{aligned}
\]

\item \label{ass:strong_convex_g} 
There exists $\underline\nu>0$ such that
$\sigma(t,s)\sigma(t,s)^\top
\succeq
\underline\nu I_m$, for all $(t,s)\in[0,T]\times\mathbb R^m$, and the function $g$ is strongly convex on $\mathbb R$. More precisely, there exists a constant $\kappa_g>0$ such that for every $w,w'\in \mathbb R$ and every $\theta\in(0,1)$,
\[
g\bigl(\theta w+(1-\theta)w'\bigr)
\le 
\theta g(w)+(1-\theta)g(w')-\frac{\kappa_g}{2}\theta(1-\theta)
\left(
|w-w'|^2
\right).
\]

\end{enumerate}
\end{proposition}

\begin{proof}
Let $\varphi,\psi\in\mathcal A$ satisfy
$J(\varphi),J(\psi)<\infty$, and let $\theta\in(0,1)$. Set
$\varphi^\theta:=\theta\varphi+(1-\theta)\psi$.
As in the proof of Proposition~\ref{prop:strict_convexity}, the
affine wealth dynamics imply
$W^{\varphi^\theta}
=
\theta W^\varphi+(1-\theta)W^\psi$.

Suppose first that Condition~\ref{ass:strong_convex_f} holds. By the
strong convexity of $f$ and the convexity of $g$,
\[
\begin{aligned}
J(\varphi^\theta)
\le{}&
\theta J(\varphi)+(1-\theta)J(\psi) \\
&-
\frac{\kappa_f}{2}\theta(1-\theta)
\E\int_0^T
\left(
|W_t^\varphi-W_t^\psi|^2
+
|\varphi_t-\psi_t|^2
\right)dt .
\end{aligned}
\]
Dropping the first nonnegative term gives
\[
J(\varphi^\theta)
\le
\theta J(\varphi)+(1-\theta)J(\psi)
-
\frac{\kappa_f}{2}\theta(1-\theta)
\|\varphi-\psi\|_{L^2_{\mathbb F}}^2.
\]
Hence $J$ is $\kappa_f$-strongly convex.

Now suppose that Condition~\ref{ass:strong_convex_g} holds. Combining~\eqref{eq:lower_bound_wealth}, 
the strong convexity of $g$ and the convexity of the running-cost
term therefore implies
\[
J(\varphi^\theta)
\le
\theta J(\varphi)+(1-\theta)J(\psi)
-
\frac{c\kappa_g}{2}\theta(1-\theta)
\|\varphi-\psi\|_{L^2_{\mathbb F}}^2,
\]
where $c:=\frac{\nu}{2}e^{-KT}.$
Thus $J$ is $c\kappa_g$-strongly convex.
\end{proof}

\begin{corollary}[Strong convergence of the computed controls]
\label{cor:strong_control_convergence}
Suppose that Assumptions~\ref{ass:market},
\ref{ass:action_set}, \ref{ass:regularity},
\ref{ass:slater}, and~\ref{ass:strong_convexity} hold. 
Then
\[
\widehat\varphi\to\varphi^\star
\qquad\text{strongly in }L^2_{\mathbb F}.
\]
\end{corollary}

\begin{proof}

Let $\lambda^\star\in\mathcal Z$. By strong duality and primal recovery,
there exists $\eta^\star\in E$ such that
$(\varphi^\star,\eta^\star)$ minimizes
$\mathcal L(\cdot,\cdot,\lambda^\star)$. Since $C$ is jointly convex,
$\mathcal L(\cdot,\cdot,\lambda^\star)$ is
$\kappa$-strongly convex in its control component. Consequently,
\[
\frac{\kappa}{2}
\|\widehat\varphi-\varphi^\star\|_{L^2_{\mathbb F}}^2
\le
\mathcal L(
\widehat\varphi,\widehat\eta,\lambda^\star)
-
q(\lambda^\star).
\]
Using $q(\lambda^\star)=p^\star$, the right-hand side equals
\[
J(\widehat\varphi)-p^\star
+
\lambda^\star C(\widehat\varphi,\widehat\eta),
\]
which converges to zero by
Theorem~\ref{thm:nested_search_convergence}. Hence
\[
\widehat\varphi\to\varphi^\star
\qquad\text{strongly in }L^2_{\mathbb F}.
\]
\end{proof}

\section{Numerical Experiments}\label{sec:numerics}

This section evaluates Algorithm~\ref{alg:lambda_bisection} in a scalar linear--quadratic Black--Scholes model. 
This benchmark isolates the dynamic effect of a terminal CVaR constraint, first in a complete market, and then in the presence of non-traded risks. 

Let the traded stock be driven by a Brownian motion $B$, and let $B^\perp$ be an independent Brownian motion driving the nontraded risk exposure.
With $r=0$,
\begin{align}\label{dyn:gbm_stock}
dS_t=S_t(\mu\,dt+\sigma\,dB_t),
\qquad \mu>0,\quad \sigma>0,\quad S_0>0,
\end{align}
and the endowment dynamic satisfies $d\zeta_t=\beta^\perp dB_t^\perp$.
With $\varphi_t$ denoting dollar risky exposure, wealth evolves as
\begin{align}\label{dyn:lq_wealth}
dW_t^\varphi
=\varphi_t\mu\,dt+\varphi_t\sigma\,dB_t
+\beta^\perp dB_t^\perp,
\qquad W_0^\varphi=w_0.
\end{align}
The CVaR-constrained portfolio management problem is therefore:
\begin{align}\label{target:lq_problem}
\inf_{\varphi\in\A}\quad
&J(\varphi)
=\E\left[\int_0^T
\left\{
\frac{\gamma}{2}\left[(\sigma\varphi_t)^2+(\beta^\perp)^2\right]
-\mu\varphi_t
\right\}dt\right]\notag\\
\text{subject to}\quad
&\operatorname{CVaR}_{\alpha}(-W_T^\varphi)\leq c.
\end{align}
For $c<0$, the constraint requires average terminal wealth in the worst $(1-\alpha)$ fraction of outcomes to be at least $-c$; it is not a pathwise wealth floor.
When the CVaR constraint is nonbinding, pointwise minimization gives the Merton dollar exposure
\begin{align}\label{unconstraint:merton_line}
\bar\varphi_t=\frac{\mu}{\gamma\sigma^2}.
\end{align}
The independent endowment shock adds variance but does not change~\eqref{unconstraint:merton_line}.
When the CVaR constraint is binding, however, the investor may reduce traded exposure to offset its effect on the lower tail.

Table~\ref{tab:numerical_common_parameters} reports the common calibration.
The nonbinding and binding CVaR limits are $c=-0.86$ and $c=-0.94$, respectively.
For fixed $(\lambda,\eta)$, the inner dynamic program computes a Markov feedback control on a discretized state grid.
We optimize $\eta$ by golden-section search and use an outer bisection in $\lambda$ to enforce the CVaR constraint.
Code and further implementation details are available at \href{https://github.com/xf-shi/Dynamic-Portfolio-under-CVaR}{https://github.com/xf-shi/Dynamic-Portfolio-under-CVaR}.

\begin{table}[H]
\centering
\caption{Common parameters used in the numerical experiments.}
\label{tab:numerical_common_parameters}
\small
\begin{tabular}{ll}
\toprule
Parameter & Value \\
\midrule
Horizon & $T=1$ \\
Risk-free rate & $r=0.00$ \\
Excess return & $\mu=0.08$ \\
Volatility & $\sigma=0.20$ \\
Initial stock price & $S_0=1$ \\
Risk aversion & $\gamma=5$ \\
Initial wealth & $w_0=1$ \\
Merton exposure & $\bar\varphi_t=0.40$ \\
CVaR confidence level & $\alpha=0.95$ \\
Nonbinding CVaR limit & $c=-0.86$ \\
Binding CVaR limit & $c=-0.94$ \\
\bottomrule
\end{tabular}
\end{table}

\subsection{Complete Market Benchmark}\label{ssec:complete}
For the complete-market benchmark, we suppress the orthogonal factor
$B^\perp$ and take the filtration to be generated by the traded
Brownian motion $B$. Equivalently, there is no nontraded endowment
risk and the single traded risky asset spans the single source of
market uncertainty.
In the complete market, the nontraded risk exposure is $\beta^\perp=0$. Table~\ref{tab:complete} shows that the constraint is nonbinding at $c=-0.8600$: the Merton strategy has terminal loss CVaR $-0.8671<-0.8600$ and $\widehat\lambda=0$.
In the binding case with $c=-0.9400$, the calibrated multiplier is $0.0720$.
Average risky exposure falls from $0.4000$ to $0.2981$, whereas expected terminal wealth declines by only about $0.8\%$.

\begin{table}[H]
\centering
\small
\begin{tabular}{lccccc}
\toprule
Case & $c$ & $\widehat\lambda$
& $\operatorname{CVaR}_{\alpha}(-W_T^\star)$
& $\E[W_T^\star]$ & average $\varphi_t^\star$ \\
\midrule
Nonbinding & $-0.8600$ & $0$ & $-0.8671$ & $1.0322$ & $0.4000$ \\
Binding & $-0.9400$ & $0.0720$ & $-0.9406$ & $1.0242$ & $0.2981$ \\
\bottomrule
\end{tabular}
\caption{Summary statistics for the complete-market benchmark in the nonbinding and binding cases of the terminal CVaR constraint.}
\label{tab:complete}
\end{table}

Figure~\ref{fig:complete_position} shows that the policy for the binding case is not a constant rescaling of the Merton exposure.
It returns toward $0.400$ along the selected favorable path, which ends at $W_T^\star\approx1.311$, but falls to about $0.068$ late in the selected unfavorable path, which ends at $W_T^\star\approx0.947$.
Thus the binding constraint induces state-dependent de-risking while retaining participation in favorable states.

\begin{figure}[H]
\centering
\includegraphics[width=0.48\linewidth]{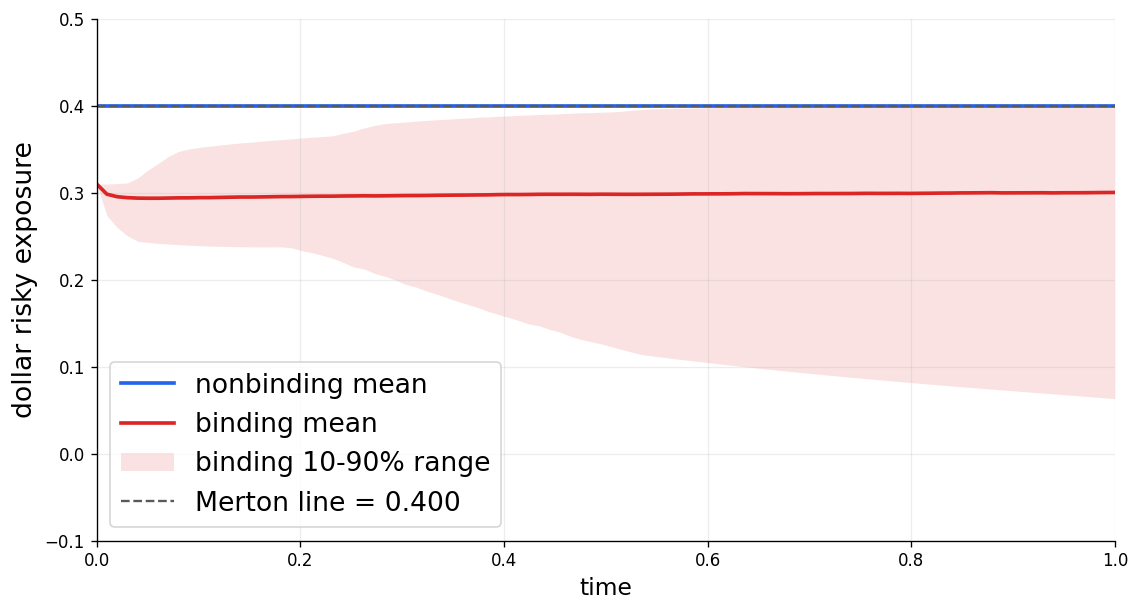}
\includegraphics[width=0.48\linewidth]{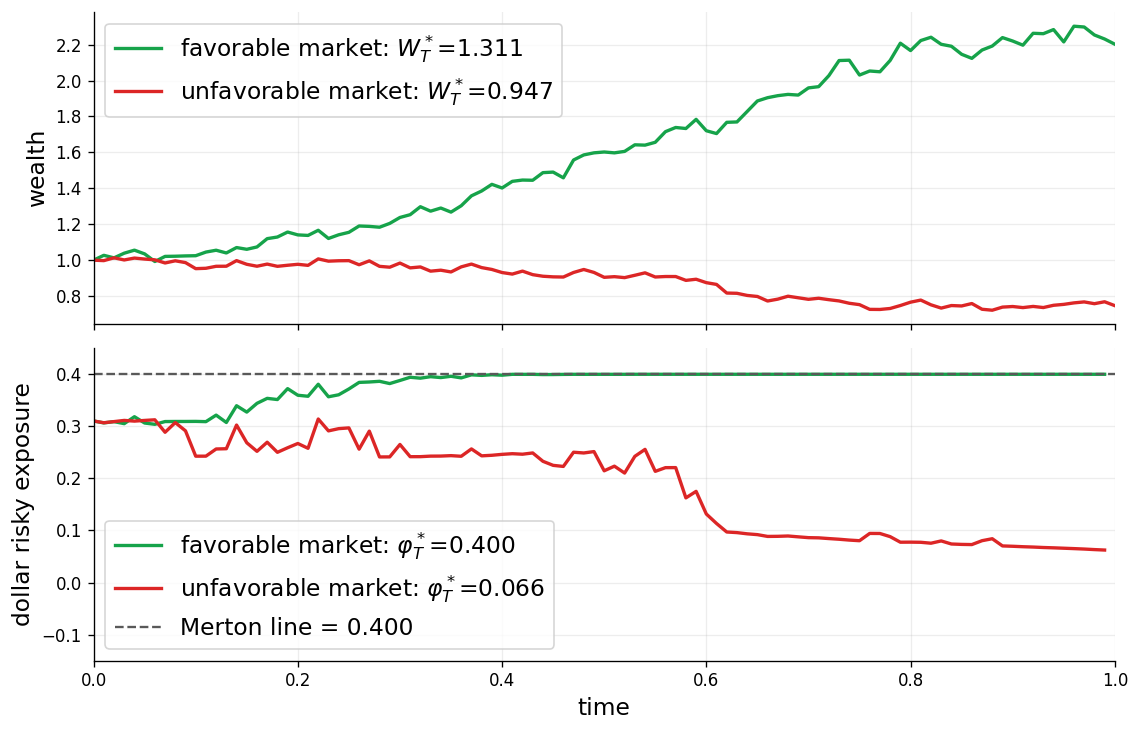}
\caption{Optimal dollar risky exposure in the complete market.
Left: mean exposure in the binding case and its 10--90\% simulation band, with the Merton exposure from the nonbinding case.
Right: selected favorable (green) and unfavorable (red) wealth paths and their associated exposures.}
\label{fig:complete_position}
\end{figure}

The policy for the binding case primarily compresses the lower tail of terminal wealth (Figure~\ref{fig:complete_wealth}).
The right panel also reports the unconditional diagnostic $\operatorname{CVaR}_\alpha(-W_t^\star)$ for $t<T$.
In this zero-rate, no-inflow calibration it remains below the numerical level $c$ throughout the horizon.
This observation is not a dynamic constraint and need not persist under other calibrations.

\begin{figure}[H]
\centering
\includegraphics[width=0.48\linewidth]{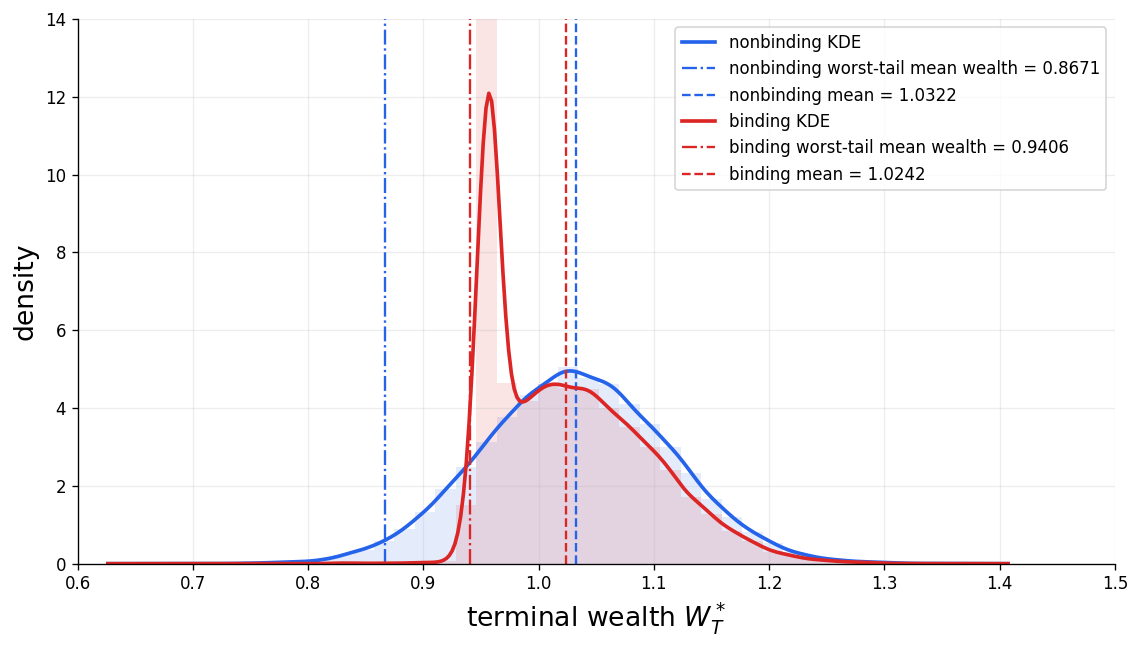}
\includegraphics[width=0.48\linewidth]{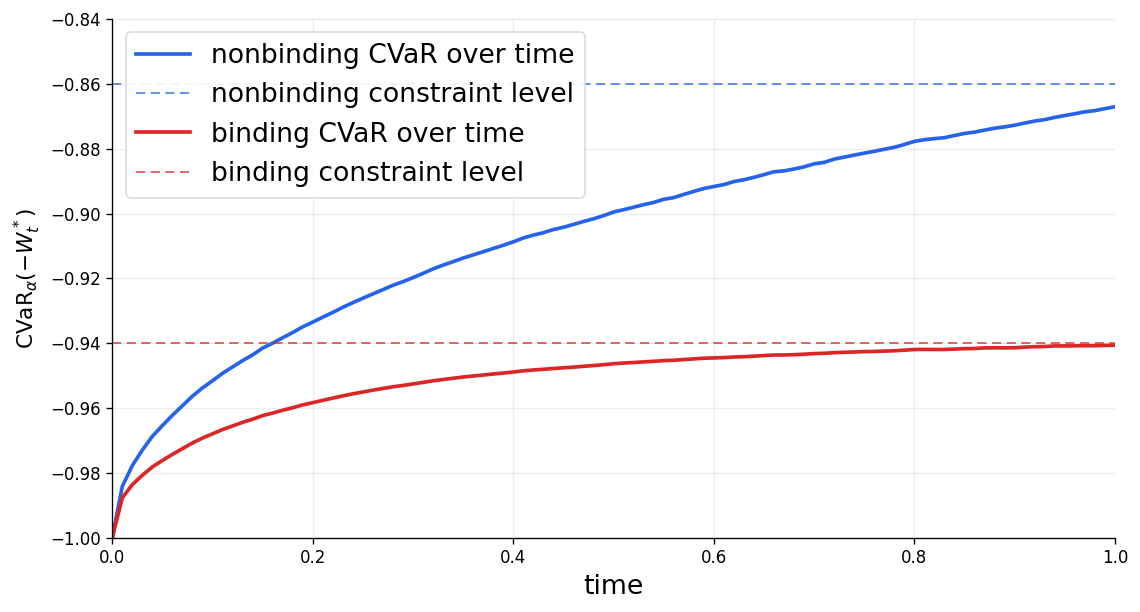}
\caption{Wealth and lower-tail risk in the complete market.
Left: terminal wealth distributions in the nonbinding and binding cases.
Right: $\operatorname{CVaR}_{\alpha}(-W_t^\star)$ as a diagnostic over time, with the terminal constraint levels.}
\label{fig:complete_wealth}
\end{figure}

This lower-tail reshaping is qualitatively related to complete-market mean-risk and quantile models~\citep{gao2017dynamic,he2015dynamic,he2011portfolio}.
Here, however, the quadratic variation penalty regularizes the solution: the binding constraint leads to smooth de-risking that depends on the state rather than to a digital or lottery-like terminal payoff.

\subsection{Incomplete Market}\label{ssec:incompete}
We set the nontraded risk exposure to be $\beta^\perp = 0.02$. 
In this incomplete market setting, the algorithm of Section~\ref{sec:algorithm} remains applicable.

In the nonbinding case, nontraded risk leaves the Merton exposure unchanged but moves terminal loss CVaR from $-0.8671$ to $-0.8620$.
In the binding case with $c=-0.9400$, the calibrated multiplier is $0.1400$ and average risky exposure falls to $0.2549$, which is $14.5\%$ below the complete-market value in the binding case and $36.3\%$ below the Merton exposure.
Expected terminal wealth falls by about $1.1\%$ relative to the policy for the nonbinding case.

\begin{table}[H]
\centering
\small
\begin{tabular}{lccccc}
\toprule
Case & $c$ & $\widehat\lambda$
& $\operatorname{CVaR}_\alpha(-W_T^\star)$
& $\E[W_T^\star]$ & average $\varphi_t^\star$ \\
\midrule
Nonbinding & $-0.86$ & $0.0000$ & $-0.8620$ & $1.0322$ & $0.4000$ \\
Binding & $-0.94$ & $0.1400$ & $-0.9405$ & $1.0206$ & $0.2549$ \\
\bottomrule
\end{tabular}
\caption{Summary statistics for the incomplete market in the nonbinding and binding cases of the terminal CVaR constraint.}
\label{tab:incomplete}
\end{table}

Figure~\ref{fig:incomplete_position} displays the resulting feedback response in the binding case.
Along the selected favorable path, wealth ends at approximately $1.237$ and exposure returns to the Merton level after $t=0.6$.
Along the selected unfavorable path, wealth ends near $0.956$ and average exposure after $t=0.6$ falls to approximately $0.043$.
The investor therefore offsets unhedgeable tail risk indirectly by reducing traded risk in unfavorable states.

\begin{figure}[H]
\centering
\includegraphics[width=0.48\linewidth]{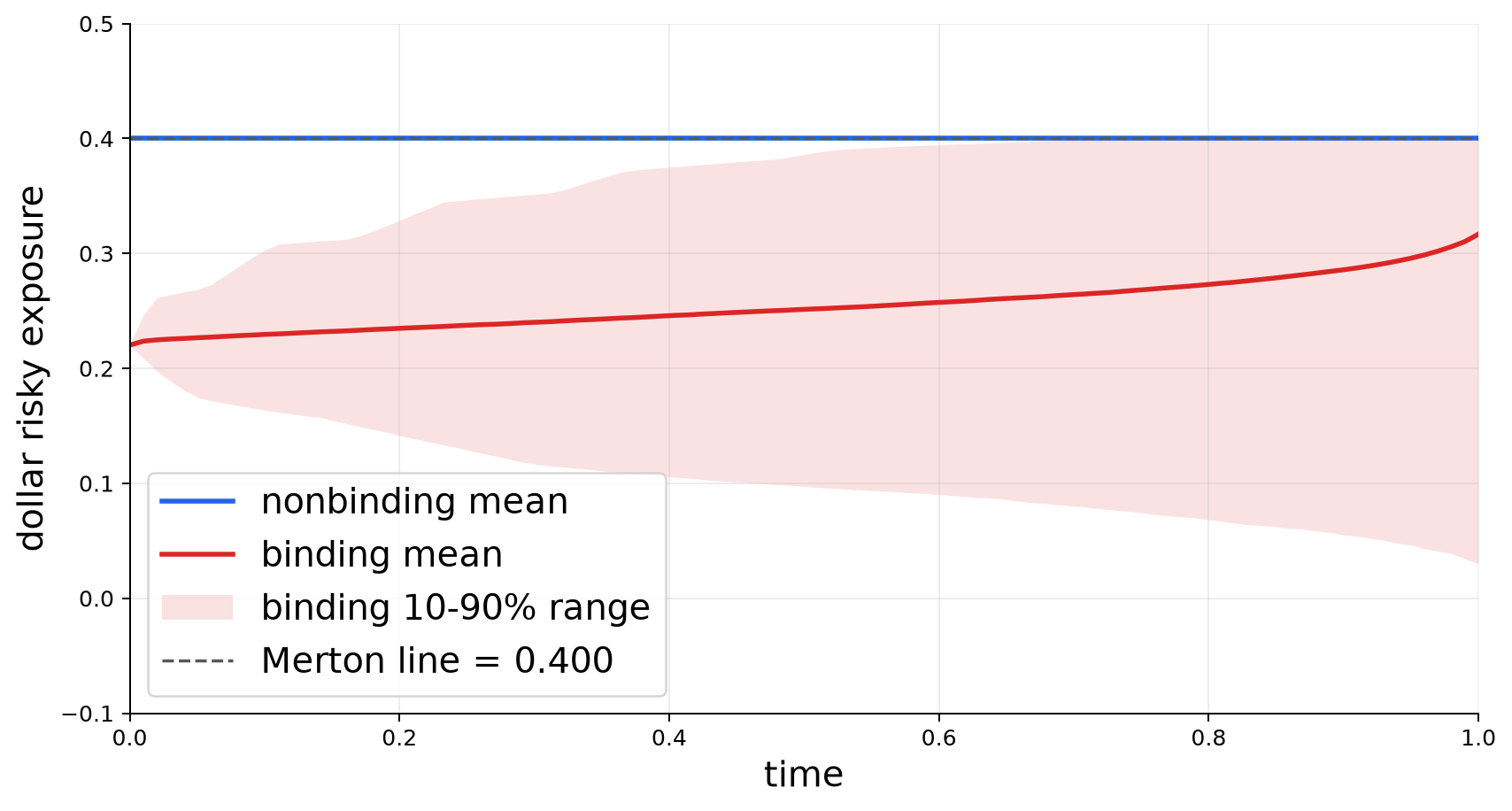}
\includegraphics[width=0.48\linewidth]{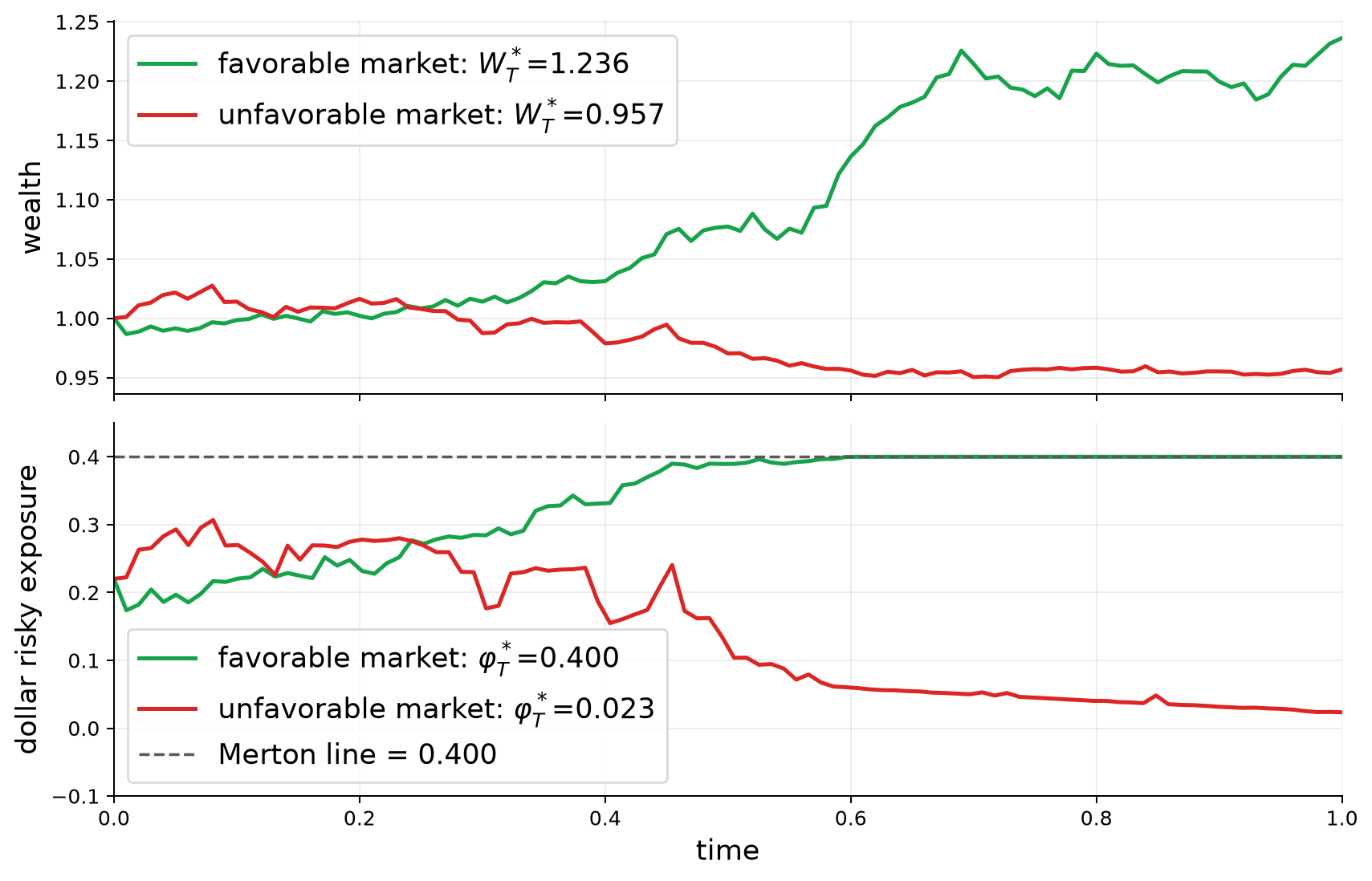}
\caption{Optimal dollar risky exposure in the incomplete market.
Left: mean exposure in the binding case and its 10--90\% simulation band, with the Merton exposure from the nonbinding case.
Right: selected favorable (green) and unfavorable (red) wealth paths and their associated exposures.}
\label{fig:incomplete_position}
\end{figure}

The distributional comparison in Figure~\ref{fig:incomplete_wealth} confirms that the policy for the binding case removes mass from severe low-wealth outcomes.
The right-panel trajectories are interim diagnostics only; the optimization imposes the CVaR restriction at $T$, not at intermediate dates.

\begin{figure}[H]
\centering
\includegraphics[width=0.48\linewidth]{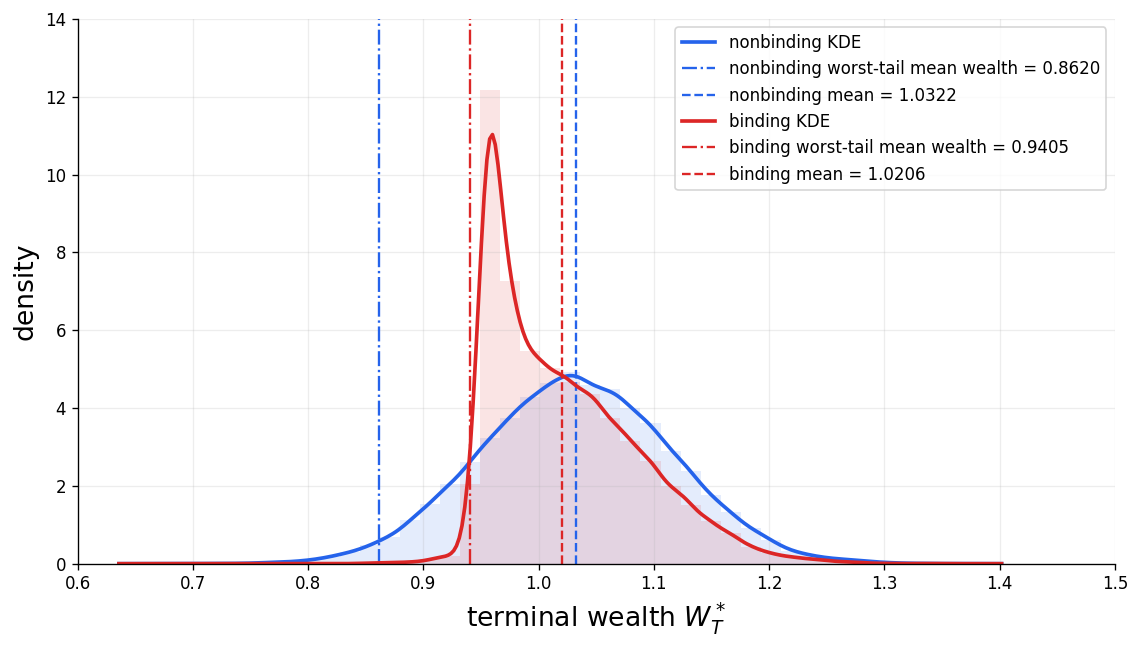}
\includegraphics[width=0.48\linewidth]{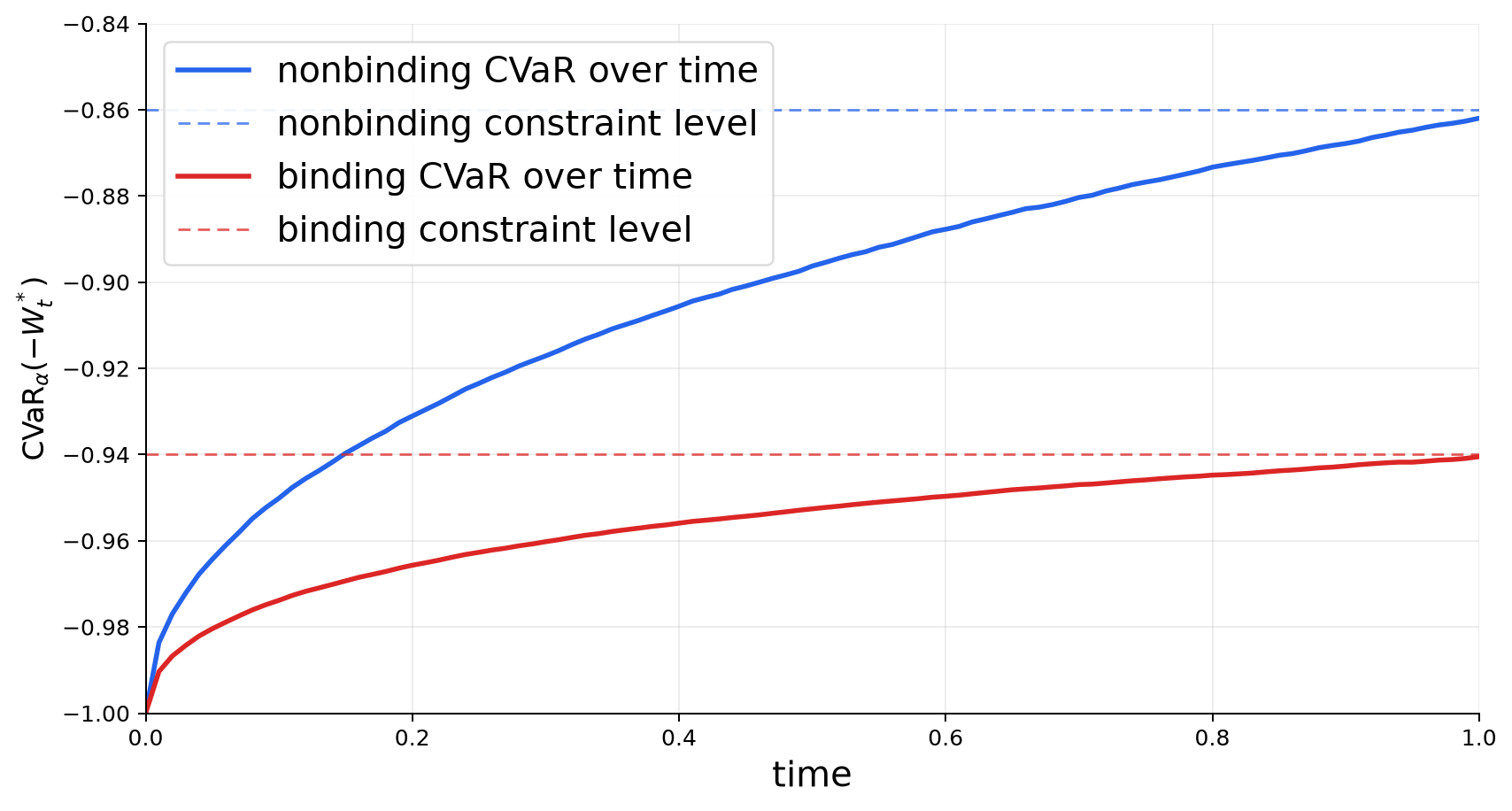}
\caption{Wealth and lower-tail risk in the incomplete market.
Left: terminal wealth distributions in the nonbinding and binding cases.
Right: $\operatorname{CVaR}_{\alpha}(-W_t^\star)$ as a diagnostic over time, with the terminal constraint levels.}
\label{fig:incomplete_wealth}
\end{figure}

\section{Portfolio Adjustment and Price Impact}\label{sec:extension}

With the presence of trading regularization or existence of price impact, the dollar risky amount becomes a state variable. 
We keep $\varphi_t$ to represent the dollar risky amount, and the control $\dot\varphi$ is referred to as the signed dollar trading rate, where the dynamic between the dollar risky amount $\varphi$ and the control $\dot\varphi$ remains linear 
\begin{align}\label{dyn:linear_control}
d\varphi_t 
&= \dot\phi_t S_t dt + \phi_t dS_t \notag\\
&= \dot\varphi_t dt +\varphi_t \left(\mu dt+\sigma dB_t \right), \qquad \varphi_0=0,
\end{align}
where we recall that $\phi_t$ stands for the total shares hold by the investor, and $\dot\phi_t$ represents her signed trading rate in shares. 
Although the results in Sections \ref{sec:setup}--\ref{sec:algorithm} are stated for a scalar wealth
state and direct portfolio controls, the same arguments extend
directly to finite-dimensional affine controlled state dynamics under
the corresponding convexity, compactness, and integrability
conditions. We use this extension in Section~\ref{ssec:linear impact}. 

The control $\dot\phi$ is used in both specifications below, but the economic interpretations and theoretical properties of the two
specifications differ. In Section~\ref{ssec:linear impact}, the quadratic term is an
objective regularizer that smooths portfolio adjustment; it is not an
execution cost and is therefore not deducted from wealth. The
augmented state dynamics remain affine in the trading-rate control, and
the quadratic regularizer preserves the convex structure underlying
the analysis in Sections~\ref{sec:reformulation} and~\ref{sec:algorithm}. In Section~\ref{ssec:nonlinear impact}, by contrast, price
impact changes the execution price and the resulting cost is deducted
directly from wealth. The resulting wealth dynamics depend nonlinearly
on the trading-rate control through $|\dot\phi_t|^{3/2}$ and therefore
fall outside the affine-state setting of Sections~\ref{sec:reformulation} and~\ref{sec:algorithm}. We use the
square-root specification as a numerical robustness experiment beyond
the setting covered by our convergence theory.

\subsection{Quadratic Trading Rate Regularization}\label{ssec:linear impact}

We first add a quadratic penalty on the signed dollar trading rate to smooth portfolio adjustment:
\begin{align}\label{target:linear_problem_original}
J(\dot\varphi)
=
\E\left[\int_0^T \left (\frac{\gamma}{2}\sigma^2\varphi_t^2+\frac{\Lambda}{2}\dot\varphi_t^2 -\mu\varphi_t\right) dt\right].
\end{align}
Since the regularization only shows up in the target functional, the wealth dynamics remain the same as~\eqref{dyn:lq_wealth}. 
The CVaR-constrained control problem is
\begin{align}\label{target:linear_problem}
\inf_{\dot\varphi\in\A}\quad
&J(\dot\varphi)\notag\\
\text{subject to}\quad
&\operatorname{CVaR}_{\alpha}(-W_T)\leq c.
\end{align}
The term $\Lambda\dot\varphi_t^2/2$ is a regularization term in the performance criterion, which penalizes abrupt changes in the risky position and produces smoother trading policies. We use the calibration from Table~\ref{tab:numerical_common_parameters} and set $\Lambda=0.01$; larger values of $\Lambda$ place greater weight on smooth portfolio adjustment.

Table~\ref{tab:linear} shows that the CVaR constraint is nonbinding at $c=-0.86$.
In the binding case, average exposure falls from $0.3104$ to $0.2415$ and $\widehat\lambda$ rises to $0.06001$.
Expected terminal wealth declines from $1.0247$ to $1.0192$, while terminal loss CVaR reaches $-0.9413$.
The policy for the binding case uses the calibrated pair $(\lambda,\eta)=(0.06001,-0.947499)$.

\begin{table}[H]
\centering
\small
\begin{tabular}{lccccc}
\toprule
Case & $c$ & $\widehat\lambda$
& $\operatorname{CVaR}_\alpha(-W_T^\star)$
& $\E[W_T^\star]$ & average $\varphi_t^\star$ \\
\midrule
Nonbinding & $-0.86$ & $0.0000$ & $-0.8971$ & $1.0247$ & $0.3104$ \\
Binding & $-0.94$ & $0.06001$ & $-0.9413$ & $1.0192$ & $0.2415$ \\
\bottomrule
\end{tabular}
\caption{Summary statistics with quadratic trading-rate regularization in the nonbinding and binding cases of the terminal CVaR constraint.}
\label{tab:linear}
\end{table}

Figure~\ref{fig:linear control} shows that both policies trade most rapidly near the initial date and that their mean trading rates approach zero near maturity.
The policy for the binding case trades less aggressively and maintains a lower exposure profile.
The exposure bands remain nondegenerate in the nonbinding case because $\dot\varphi$ contains the stock-price diffusion in \eqref{dyn:linear_control}.

\begin{figure}[H]
\centering
\includegraphics[width=0.48\linewidth]{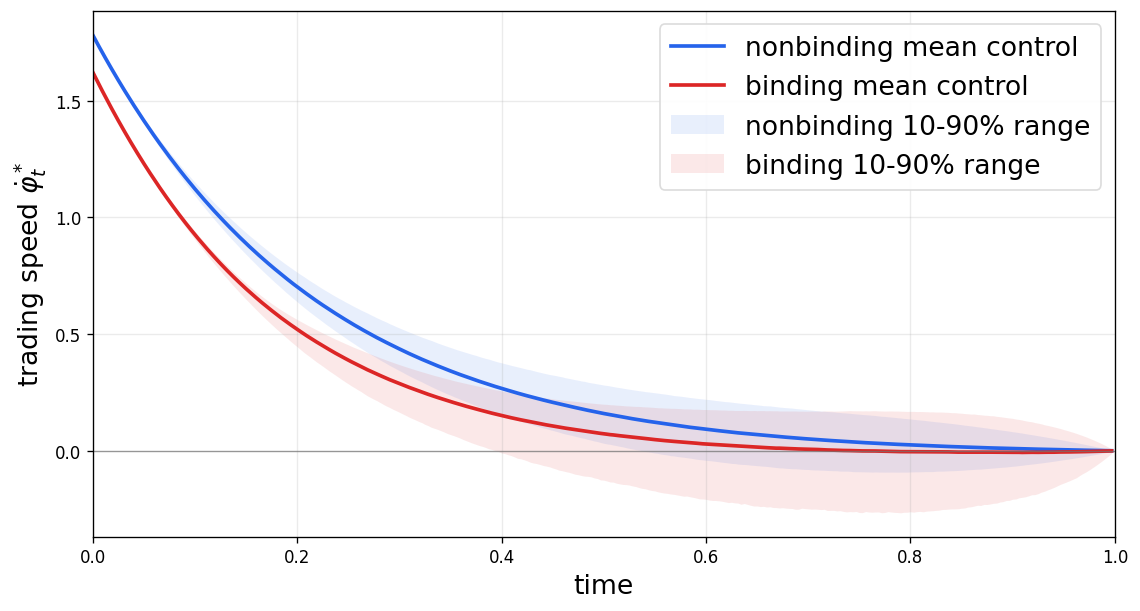}
\includegraphics[width=0.48\linewidth]{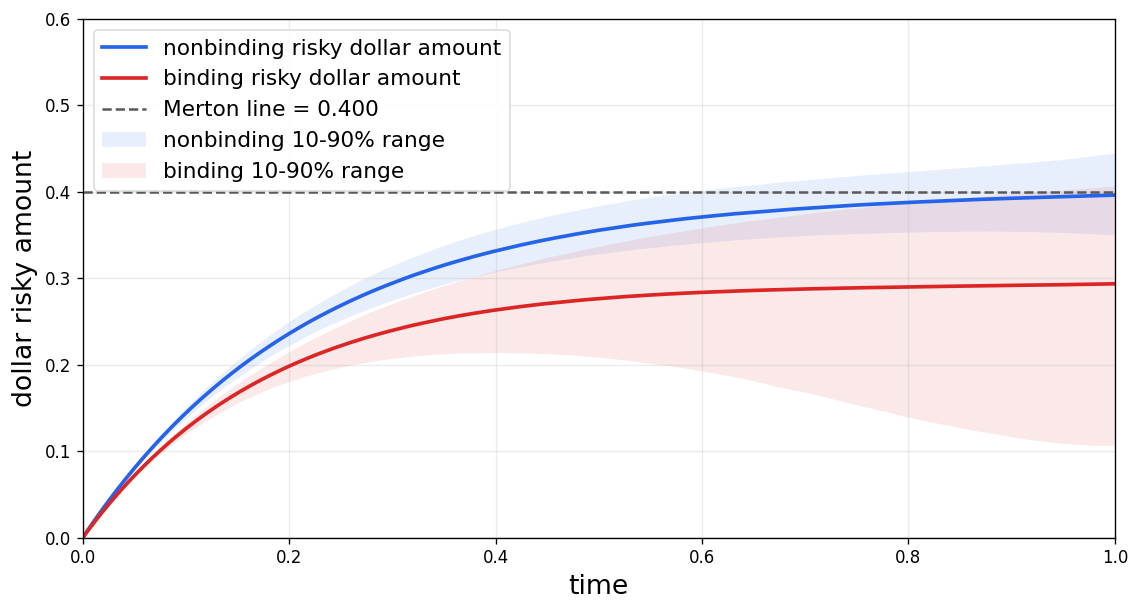}
\caption{Signed dollar trading rate and dollar risky exposure under the quadratically regularized specification.
Left: mean trading rates and 10--90\% simulation bands.
Right: associated exposure profiles and bands.}
\label{fig:linear control}
\end{figure}

In the binding case, terminal wealth is more concentrated and has less downside mass (Figure~\ref{fig:linear wealth}).
The right panel reports $\operatorname{CVaR}_\alpha(-W_t)$ as an interim diagnostic.
The constraint is imposed at maturity.

\begin{figure}[H]
\centering
\includegraphics[width=0.48\linewidth]{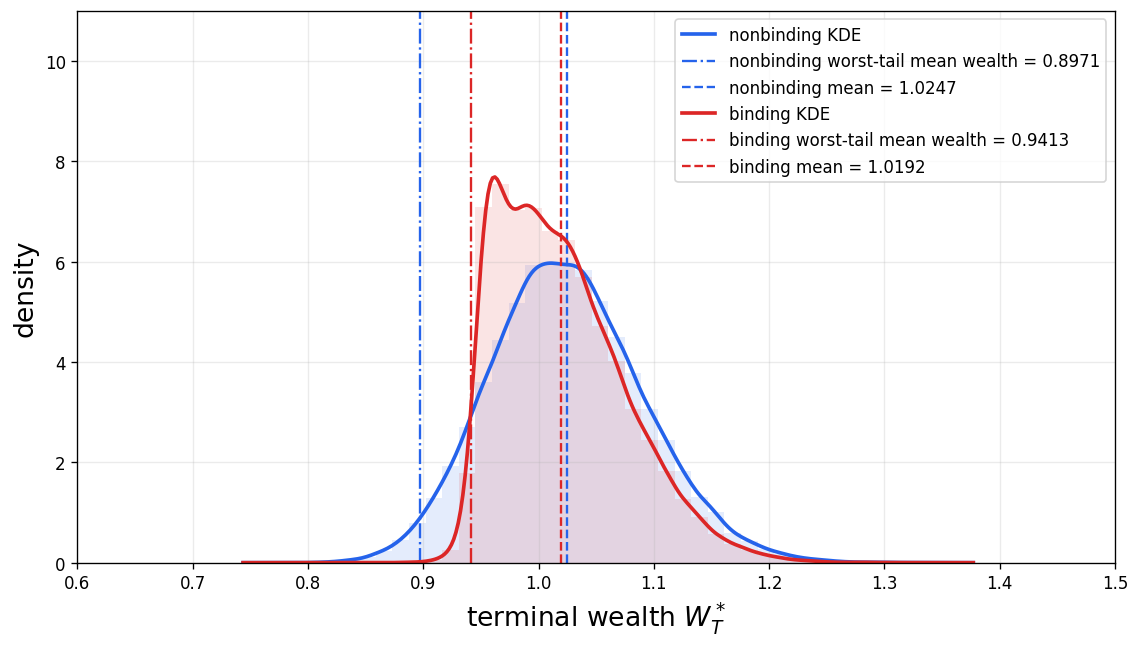}
\includegraphics[width=0.48\linewidth]{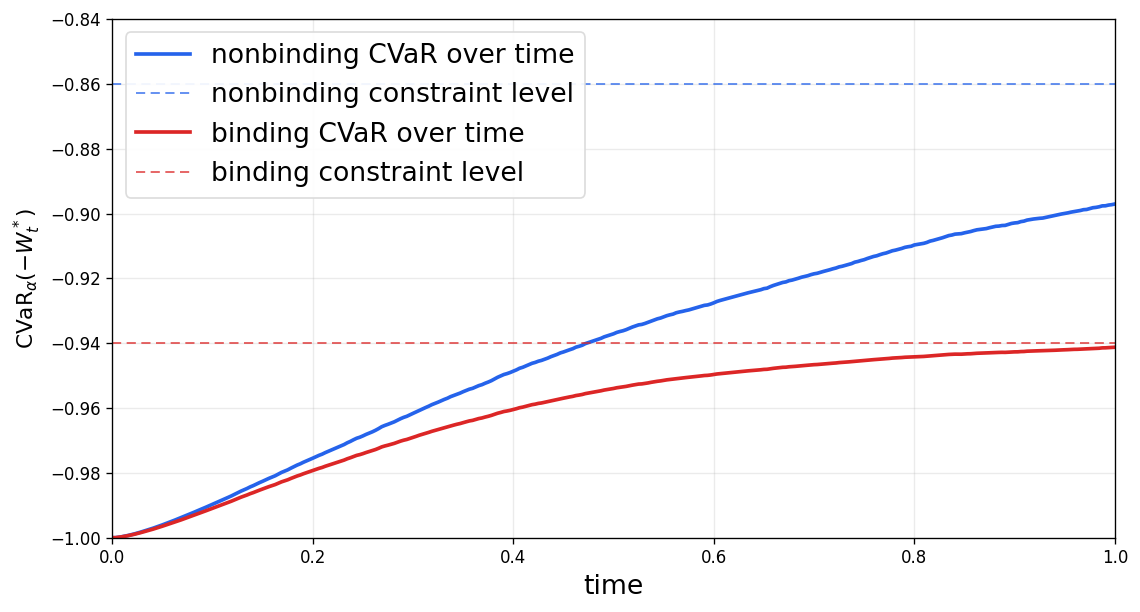}
\caption{Wealth and lower-tail risk under quadratic trading-rate regularization.
Left: terminal-wealth distributions.
Right: out-of-sample $\operatorname{CVaR}_\alpha(-W_t)$ diagnostics and terminal constraint levels.}
\label{fig:linear wealth}
\end{figure}

The selected paths in Figure~\ref{fig:linear sample} illustrate the feedback mechanism under the binding constraint.
The favorable path ends at wealth $1.085$ and approaches the Merton exposure, whereas the unfavorable path ends near $0.942$ and reduces exposure to approximately $0.064$ late in the horizon.
The trading rates along both paths decrease toward zero as maturity approaches. 
In the nonbinding case, the trading rate converges to zero as maturity approaches because further repositioning offers no remaining expected-return benefit while still incurring the quadratic trading-rate penalty.
In the binding case, the last reported preterminal rate remains slightly nonzero because the terminal-loss term associated with the CVaR constraint creates an additional incentive to adjust exposure over the final decision interval.

\begin{figure}[H]
\centering
\includegraphics[width=0.8\linewidth]{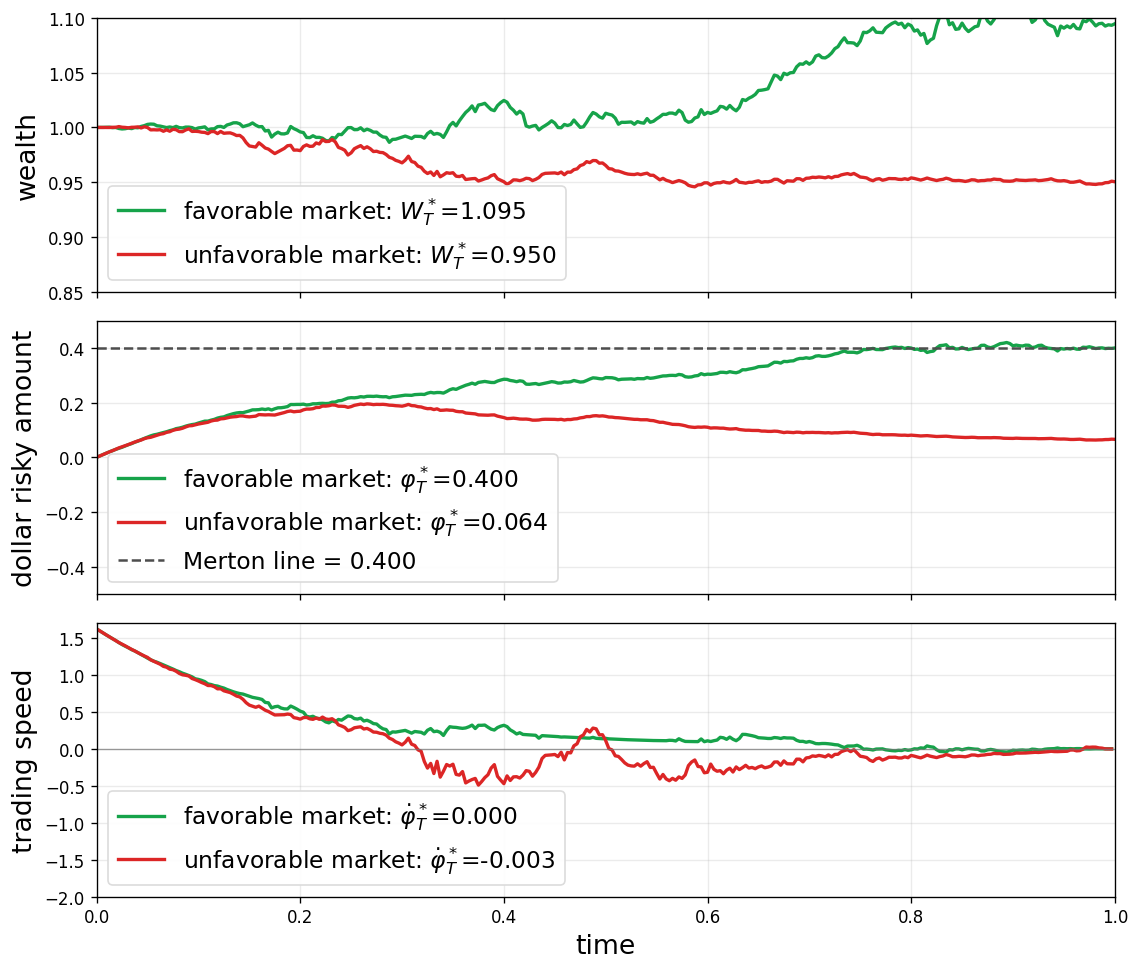}
\caption{Wealth, dollar risky exposure, and signed dollar trading rate along selected favorable (green) and unfavorable (red) realizations in the binding case.}
\label{fig:linear sample}
\end{figure}

\subsection{Square-Root Price Impact}\label{ssec:nonlinear impact}
We retain the same market and initial position, but now let execution price respond to signed dollar trading rate according to the square-root specification. 
When executing a stock, the execution price $S^{\text{exe}}$ is usually different from the market price $S_t$, i.e.
\begin{align}\label{dyn:square-root price impact}
S_t^{\mathrm{exe}}
=S_t\left[
1+\frac{\Lambda_{3/2}}{3/2}
\operatorname{sign}(\dot\varphi_t)\sqrt{|\dot\varphi_t|}
\right].
\end{align}
The placement of the brackets in \eqref{dyn:square-root price impact} ensures that $S_t^{\mathrm{exe}}=S_t$ when $\dot\varphi_t=0$.
Since $\dot\varphi_t=\dot\phi_t S_t$, execution costs reduce wealth as follows:
\begin{align}
dW_t
&=\phi_t\,dS_t-(S_t^{\mathrm{exe}}-S_t)\dot\phi_t\,dt\notag\\
&=\varphi_t(\mu\,dt+\sigma\,dB_t)
-\frac{\Lambda_{3/2}}{3/2}|\dot\varphi_t|^{3/2}\,dt.
\end{align}
The corresponding portfolio management problem is 
\begin{align}\label{target:power_problem_original}
\inf_{\dot\varphi\in\A}\quad
&\E\left[\int_0^T
\left(
\frac{\gamma}{2}\sigma^2\varphi_t^2
+\frac{\Lambda_{3/2}}{3/2}|\dot\varphi_t|^{3/2}
-\mu\varphi_t
\right)dt\right]\notag\\
\text{subject to}\quad
&\operatorname{CVaR}_{\alpha}(-W_T)\leq c.
\end{align}
Although the cost $|\dot\varphi|^{3/2}$ is convex, its nonlinear
appearance in the wealth dynamics places this specification outside
the affine-state setting of Sections~\ref{sec:setup}--\ref{sec:algorithm}. The convergence guarantees
established above therefore do not apply directly.
We set $\Lambda_{3/2}=0.0045$ and use the remaining parameters from Table~\ref{tab:numerical_common_parameters}.

In the nonbinding case, terminal loss CVaR is $-0.8839$ and the multiplier is zero.
In the binding case, the policy reduces average exposure from $0.3500$ to $0.2422$ and mean terminal wealth from $1.0260$ to $1.0180$, as shown in Table~\ref{tab:power}.

\begin{table}[H]
\centering
\small
\begin{tabular}{lccccc}
\toprule
Case & $c$ & $\widehat\lambda$
& $\operatorname{CVaR}_\alpha(-W_T^\star)$
& $\E[W_T^\star]$ & average $\varphi_t^\star$ \\
\midrule
Nonbinding & $-0.86$ & $0.0000$ & $-0.8839$ & $1.0260$ & $0.3500$ \\
Binding & $-0.94$ & $0.07818$ & $-0.9409$ & $1.0180$ & $0.2422$ \\
\bottomrule
\end{tabular}
\caption{Summary statistics for square-root price impact in the nonbinding and binding cases of the terminal CVaR constraint.}
\label{tab:power}
\end{table}

Figure~\ref{fig:power_control} shows faster initial accumulation under the policy for the nonbinding case.
Under the policy for the binding case, later trading rates can become negative in unfavorable states, indicating partial liquidation, while mean exposure stabilizes near $0.24$.
Diffusion in \eqref{dyn:linear_control} generates widening exposure bands under both policies.

\begin{figure}[H]
\centering
\includegraphics[width=0.48\linewidth]{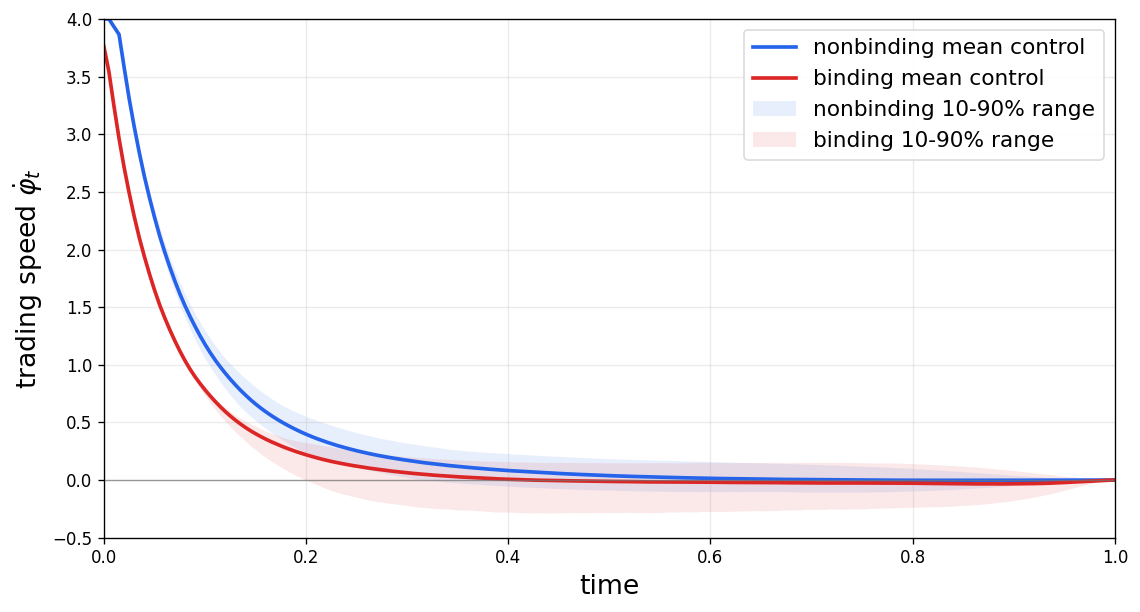}
\includegraphics[width=0.48\linewidth]{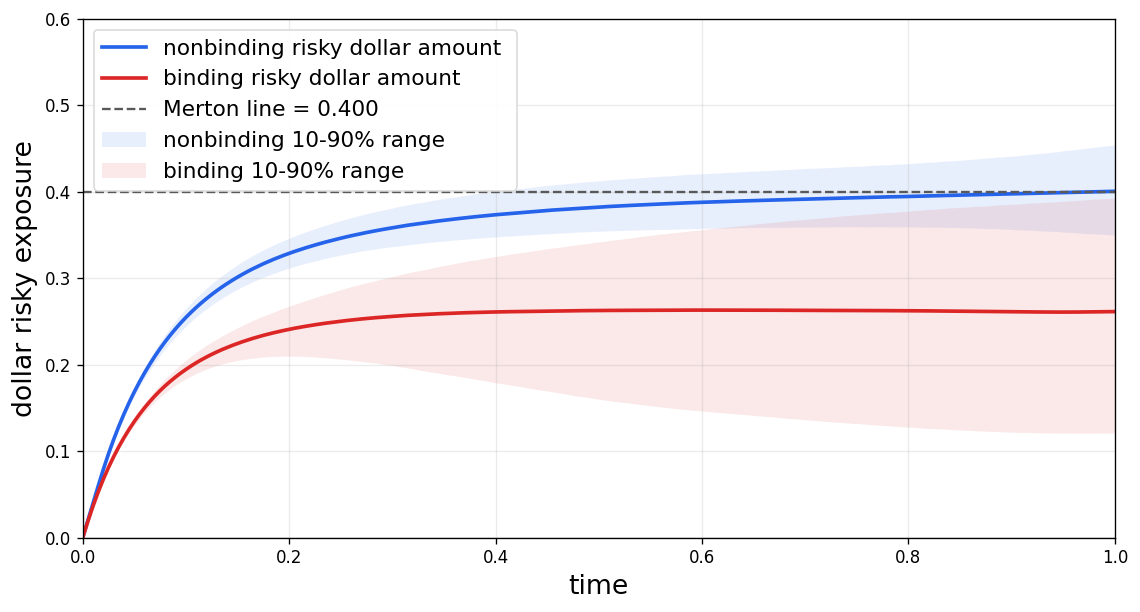}
\caption{Trading and exposure under square-root price impact.
Left: mean signed dollar trading rates and 10--90\% simulation bands.
Right: associated dollar risky exposures and bands, with the Merton reference.}
\label{fig:power_control}
\end{figure}

The policy for the binding case compresses the terminal-wealth distribution and reduces extreme outcomes (Figure~\ref{fig:power_state}).
Out-of-sample terminal loss CVaR is $-0.9409$, below the required level $-0.9400$.

\begin{figure}[H]
\centering
\includegraphics[width=0.48\linewidth]{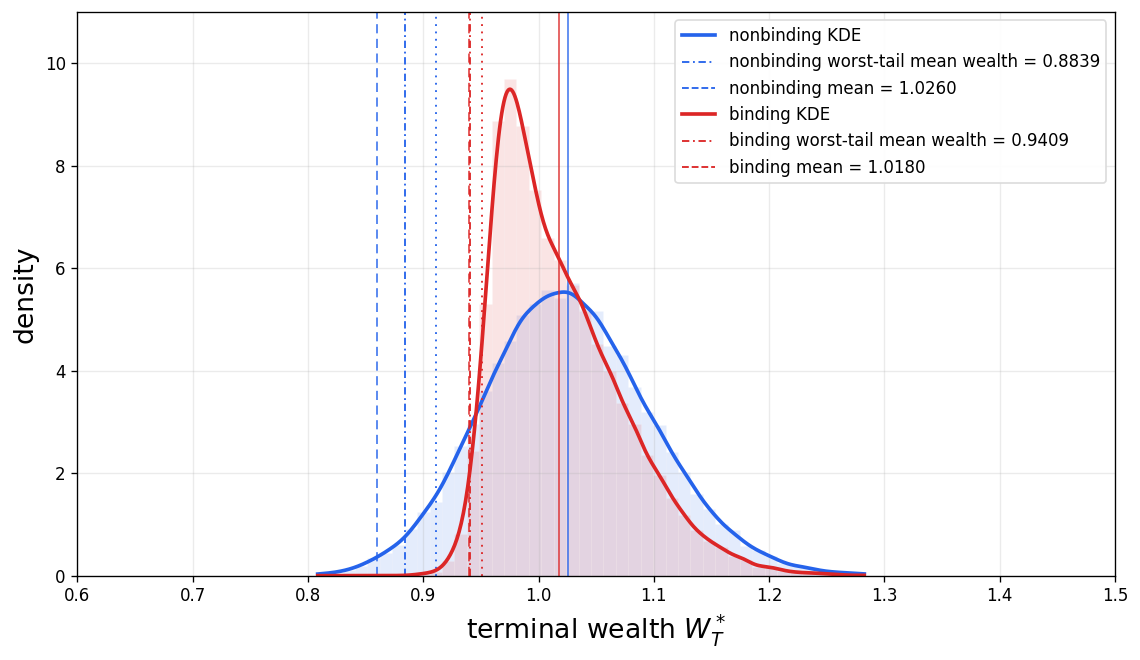}
\includegraphics[width=0.48\linewidth]{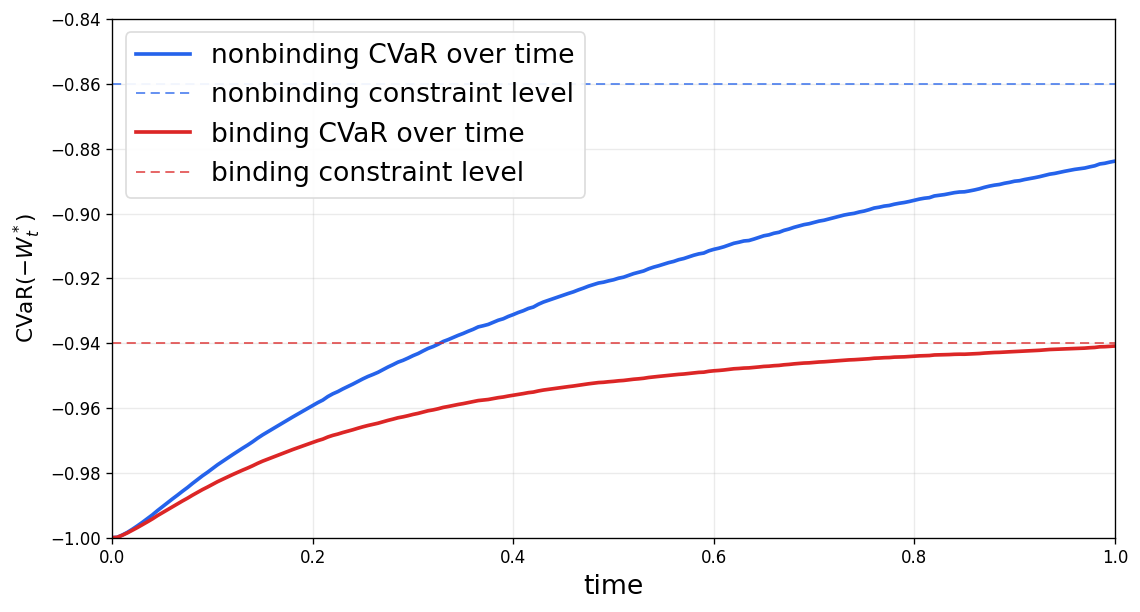}
\caption{Wealth and tail risk under $3/2$-power trading costs.
Left: terminal wealth distributions.
Right: out-of-sample loss-CVaR diagnostics and terminal constraint levels.}
\label{fig:power_state}
\end{figure}

In Figure~\ref{fig:power_sample}, the selected favorable path under the binding constraint ends at wealth $1.112$ and returns toward the Merton exposure.
The selected unfavorable path ends near $0.947$ and reduces late-horizon exposure to approximately $0.097$; its trading rate is negative over several intervals.
Trading rates approach zero near maturity in both paths.

\begin{figure}[H]
\centering
\includegraphics[width=0.8\linewidth]{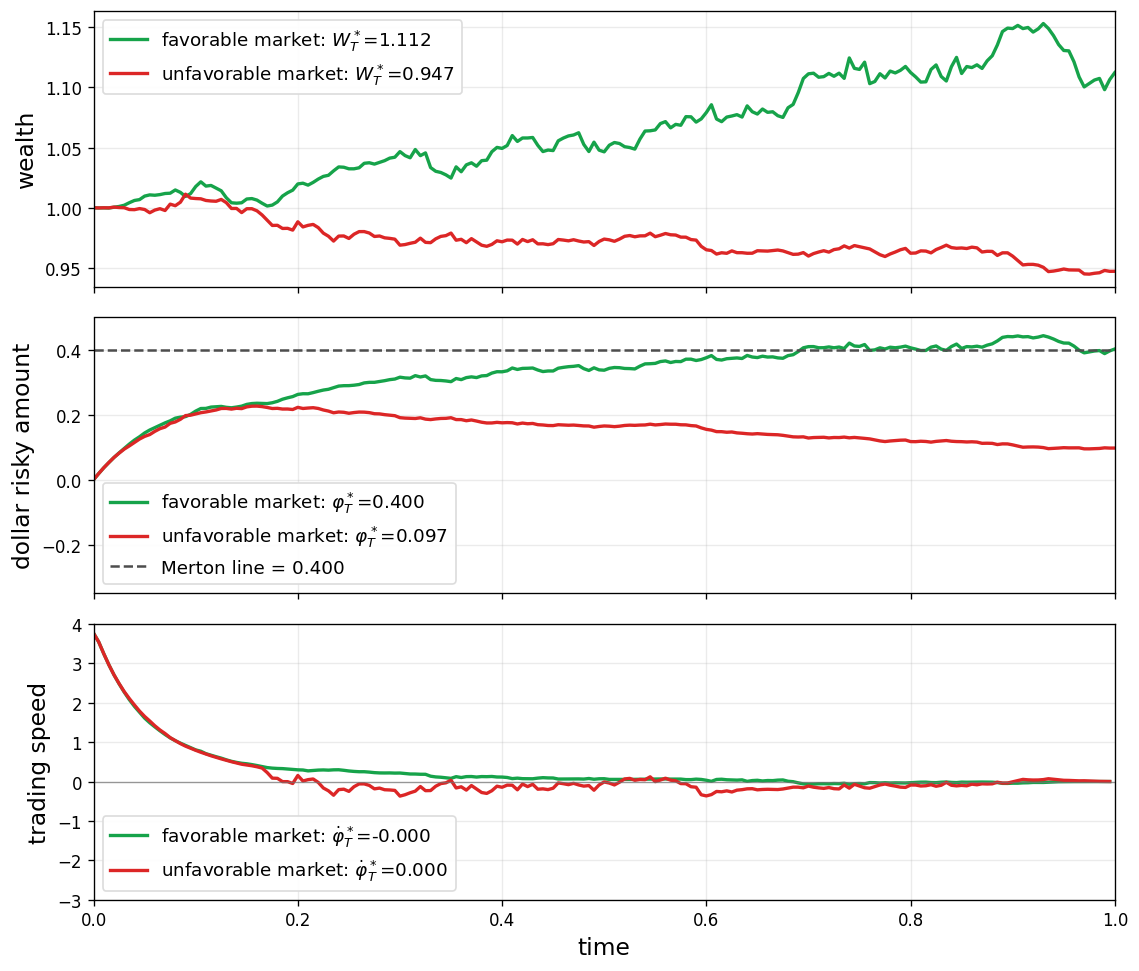}
\caption{Wealth, dollar risky exposure, and signed dollar trading rate along selected favorable (green) and unfavorable (red) realizations in the binding case.}
\label{fig:power_sample}
\end{figure}

The $3/2$-power specification produces a mean exposure in the binding case close to that under quadratic regularization but a different adjustment pattern: it permits faster initial accumulation and uses state-dependent sales when downside risk increases.
Its held-out feasibility provides numerical evidence that the method remains effective in this nonlinear price-impact specification, although the convergence guarantees do not apply
directly.

\section{Conclusion}\label{sec:conclusion}
This paper studies a dynamic portfolio management problem under an explicit CVaR constraint on terminal loss. Using the auxiliary-threshold representation of CVaR, we obtain a convex formulation that accommodates general market settings, including incomplete market, nontraded endowment risks, and general convex trading objectives. We establish existence and strong duality and develop a modular numerical method that combines standard unconstrained stochastic-control problems with one-dimensional searches over the CVaR threshold and Lagrange multiplier, which is shown to converge globally.

The numerical experiments show that a binding terminal CVaR constraint does not induce uniform de-risking. Instead, the optimal policy reduces risky exposure following adverse outcomes while preserving participation in favorable states. Nontraded endowment risk strengthens this adjustment when the constraint binds, whereas trading frictions lower the desired exposure and slow portfolio adjustment. Across the environments considered, stronger lower-tail protection is achieved at a comparatively modest cost in expected terminal wealth. These findings illustrate how a terminal risk constraint can shape state-dependent portfolio decisions throughout the investment horizon.

\section*{Acknowledgments}
SP and XS acknowledge financial support from the Natural Sciences and Engineering Research Council of Canada (RGPIN-2025-05847, RGPIN-2024-04569). AH thanks the Fields Institute for the FOCUS visitor support program and acknowledges financial support by InnoHK initiative, The Government of the HKSAR, and Laboratory for AI-Powered Financial Technologies.

\bibliographystyle{abbrv}
\bibliography{reference.bib}
\end{document}